\documentclass{article}
\usepackage{graphicx}
\usepackage{amssymb}
\usepackage{amsmath}
\usepackage{amsthm}
\usepackage{color}
\newtheorem{theorem}{Theorem}[section]

\newtheorem{corollary}[theorem]{Corollary}
\newtheorem{proposition}[theorem]{Proposition}
\theoremstyle{definition}
\newtheorem{remark}[theorem]{Remark}

\newtheorem{definition}[theorem]{Definition}
\numberwithin{equation}{section}

\newcommand{\ignore}[1]{}
\newcommand{\bli}{\begin{list}{}{\labelwidth6mm\leftmargin8mm}}
\newcommand{\eli}{\end{list}}

\newcommand{\dint}{\;\mathrm{d}}

\newcommand{\fz}{\infty}
\def\ls{\lesssim}
\def\vz{\varphi}
\def\supp{{\mathop\mathrm{supp\,}\nolimits}}

\newcommand{\Ae}{{A}_{p_1,q_1}^{s_1}}
\newcommand{\Az}{{A}_{p_2,q_2}^{s_2}}
\newcommand{\bt}{{B}_{p,q}^{s,\tau}}
\newcommand{\bte}{{B}_{p_1,q_1}^{s_1,\tau_1}}
\newcommand{\btz}{{B}_{p_2,q_2}^{s_2,\tau_2}}
\newcommand{\at}{{A}_{p,q}^{s,\tau}}
\newcommand{\ate}{{A}_{p_1,q_1}^{s_1,\tau_1}}
\newcommand{\atz}{{A}_{p_2,q_2}^{s_2,\tau_2}}
\newcommand{\ft}{{F}_{p,q}^{s,\tau}}
\newcommand{\fte}{{F}_{p_1,q_1}^{s_1,\tau_1}}
\newcommand{\ftz}{{F}_{p_2,q_2}^{s_2,\tau_2}}
\newcommand{\MA}{\ensuremath{{\cal A}^{s}_{u,p,q}}}
\newcommand{\MAe}{\ensuremath{{\cal A}^{s_1}_{u_1,p_1,q_1}}}
\newcommand{\MAz}{\ensuremath{{\cal A}^{s_2}_{u_2,p_2,q_2}}}
\newcommand{\MB}{\ensuremath{{\cal N}^{s}_{u,p,q}}}
\newcommand{\MBe}{\ensuremath{{\cal N}^{s_1}_{u_1,p_1,q_1}}}
\newcommand{\MBz}{\ensuremath{{\cal N}^{s_2}_{u_2,p_2,q_2}}}
\newcommand{\MF}{\ensuremath{{\cal E}^{s}_{u,p,q}}}
\newcommand{\MFe}{\ensuremath{{\cal E}^{s_1}_{u_1,p_1,q_1}}}
\newcommand{\MFz}{\ensuremath{{\cal E}^{s_2}_{u_2,p_2,q_2}}}
\newcommand{\M}{\ensuremath{{\cal M}_{u,p}}}
\newcommand{\cM}{\ensuremath{{\mathcal M}}}
\newcommand{\SRn}{\mathcal{S}(\rn)}
\newcommand{\SpRn}{\mathcal{S}'(\rn)}

\newcommand{\bmo}{\mathrm{bmo}}
\newcommand{\Lloc}{L_1^{\mathrm{loc}}}
\newcommand{\B}{\ensuremath{B^s_{p,q}}}
\newcommand{\F}{\ensuremath{F^s_{p,q}}}
\newcommand{\A}{\ensuremath{A^s_{p,q}}}
\newcommand{\beq}{\begin{equation}}
\newcommand{\eeq}{\end{equation}}
\newcommand{\ve}{\varepsilon}

\newcommand{\id}{\ensuremath{\operatorname{id}}}

\newcommand{\real}{\ensuremath{{\mathbb R}}}
\newcommand{\rr}{{\real}}
\newcommand{\rn}{\ensuremath{{\real^d}}}

\newcommand{\zz}{\ensuremath{{\mathbb Z}}}
\newcommand{\zn}{\ensuremath{{\zz^d}}}
\newcommand{\nn}{\ensuremath{{\mathbb N}}}

\newcommand{\cs}{\ensuremath{\mathcal S}}

\newcommand{\critical}{\ensuremath{\gamma(\tau_1,\tau_2,p_1,p_2)}}

\newcommand{\ext}{\ensuremath{\mathop\mathrm{ext}\nolimits}}
\newcommand{\re}{\ensuremath{\mathop\mathrm{re}}}

\usepackage{cite}

\usepackage[symbol,multiple]{footmisc}

\begin{document}

\title{Compact embeddings in Besov-type and Triebel-Lizorkin-type Spaces on bounded domains}

\author{Helena F. Gon\c{c}alves\footnotemark[1], Dorothee D. Haroske\footnotemark[1], and Leszek Skrzypczak\footnotemark[1]\ \footnotemark[2]}







\maketitle

\footnotetext[1]{All authors were partially supported by the German Research Foundation (DFG), Grant no. Ha 2794/8-1.}

\footnotetext[2]{The author was supported by National Science Center, Poland,  Grant no. 2013/10/A/ST1/00091.}

\begin{abstract} 
{We study embeddings of Besov-type and Triebel-Lizorkin-type spaces, $\id_\tau~:~\bte(\Omega)~\hookrightarrow~\btz(\Omega)$ and $\id_\tau : \fte(\Omega)  \hookrightarrow \ftz(\Omega) $, where $\Omega \subset \rn$ is a bounded domain, and obtain necessary and sufficient conditions for the compactness of $\id_\tau$. Moreover, we characterise its entropy and approximation numbers. Surprisingly, these results are completely obtained via embeddings and the application of the corresponding results for classical Besov and Triebel-Lizorkin spaces as well as for Besov-Morrey and Triebel-Lizorkin-Morrey spaces.} 
\end{abstract}

\section{Introduction}

{Usually called \textit{smoothness spaces of Morrey type} or, for short, \textit{smoothness Morrey spaces}, these function spaces are built upon Morrey spaces $\M(\rn)$, $0< p \le u < \infty$, and attracted some attention in the last decades, motivated firstly by possible applications. They include Besov-Morrey spaces $\MB(\rn)$, Triebel-Lizorkin-Morrey spaces $\MF(\rn)$, $0 < p \le u < \infty$, $0 < q \le \infty$, $s \in \real$, Besov-type spaces $\bt(\rn)$ and Triebel-Lizorkin-type spaces $\ft(\rn)$, $0 < p <\infty$, $0 < q \le \infty$, $\tau\ge 0$, $s \in \real$.

The classical Morrey spaces $\M$, $0< p \le u < \infty$ , were introduced by Morrey in \cite{Mor} and are part of a wider class of Morrey-Campanato spaces, cf. \cite{Pee}. They can be seen as a complement to $L_p$ spaces, since $\ensuremath{{\cal M}_{p,p}}(\rn)= L_p(\rn)$. 

The Besov-Morrey spaces $\MB(\rn)$ were introduced by Kozono and Yamazaki in \cite{KY} and used by them and later on by Mazzucato \cite{Maz} in the study of Navier-Stokes equations. In \cite{TX} Tang and Xu introduced the corresponding Triebel-Lizorkin-Morrey spaces $\MF(\rn)$, thanks to establishing the Morrey version of Fefferman-Stein vector-valued inequality. Some properties of these spaces including their wavelet characterisations were later described in the papers by Sawano \cite{Saw2,Saw1}, Sawano and Tanaka \cite{ST2,ST1} and Rosenthal \cite{MR-1}. Recently, some limiting embedding properties of these spaces were investigated in a series of papers \cite{hs12,hs12b,hs14, HaSk-krakow}.

Another class of generalisations, the Besov-type space $\bt(\rn)$ and the Triebel-Lizorkin-type space $\ft(\rn)$ were introduced in \cite{ysy}. Their homogeneous versions were originally investigated by  El Baraka in \cite{ElBaraka1,ElBaraka2, ElBaraka3} and by Yuan and Yang \cite{yy1,yy2}. There are also some applications in partial differential equations for spaces of type $\bt(\rn)$ and $\ft(\rn)$, such as (fractional) Navier-Stokes equations, cf. \cite{lxy}.

Although the above scales are defined in different ways, they share some properties and are related to each other by a number of embeddings and coincidences. For instance, they both include the classical spaces of type $\B(\rn)$ and $\F(\rn)$ as special cases. We refer to our papers mentioned above, to the recently published papers \cite{YHMSY, YHSY}, but in particular to the fine surveys \cite{s011,s011a} by Sickel.

There is still a third approach, due to Triebel, who introduced and studied in \cite{t13} local spaces and in \cite{t14} hybrid spaces, together with their use in heat equations and Navier-Stokes equations. However, since the hybrid spaces coincide with appropriately chosen spaces of type $\bt(\rn)$ or $\ft(\rn)$, respectively, cf. \cite{ysy2}, we do not have to deal with them separately now.

In this paper we investigate the compactness of the embeddings of the spaces $\bt(\Omega)$ and $\ft(\Omega)$, where $\Omega \subset \rn$ is a bounded smooth domain. In particular, our first goal is to find necessary and sufficient conditions for the compactness of the embeddings
\begin{equation}\label{e0.0}
\id_{\tau}: \ate(\Omega) \hookrightarrow \atz(\Omega),
\end{equation}
where $A = B$ or $A= F$, cf. Theorem \ref{comp-tau-u}. Here we prove that $\id_\tau$ is compact if, and only if, 
\begin{equation}\label{e0.1}
\frac{s_1-s_2}{d}>\max\left\{\left(\tau_2-\frac{1}{p_2}\right)_+ -\left(\tau_1-\frac{1}{p_1}\right)_+, \frac{1}{p_1}-\tau_1 - \min\left\{\frac{1}{p_2}-\tau_2, \frac{1}{p_2}(1-p_1\tau_1)_+\right\}\right\},
\end{equation}
where we use the notation $a_+:= \max\{a,0\}$. At this point, this work can be seen as a counterpart of the papers \cite{hs12b,hs14, HaSk-krakow}, where we studied the compactness of the corresponding embeddings of the spaces $\MB$ and $\MF$. 

Usually one would start by studying the continuity of such embeddings and later proceed to the compactness. Here we do it differently and start by dealing with the compactness. Our technique relies basically on embeddings. Since for compactness one always has strict inequalities, like condition \eqref{e0.1}, one can always have further embeddings in between the considered spaces. Therefore, we take advantage of the relations between this scale, the smoothness Morrey spaces $\MB$ and $\MF$ and the classical spaces of type $\B$ and $\F$, and use the corresponding results for these spaces to obtain our main result. 

Afterwards we qualify the compactness of $\id_\tau$ in \eqref{e0.0} by means of entropy and approximation numbers. In the recent works \cite{HaSk-krakow} and \cite{HaSk-morrey-comp}, we characterised entropy and approximation numbers of the embedding
\[
\id_{\mathcal{A}}: \MAe(\Omega) \hookrightarrow \MAz(\Omega),
\]
with $\mathcal{A}= \mathcal{N}$ or $\mathcal{A}= \mathcal{E}$. However, to the best of our knowledge, apart from a result obtained in \cite{YHMSY} for approximation numbers when the target space is $L_{\infty}$, nothing is known on this matter for embeddings between spaces of type $\at$. Here we contribute a little more to the development of this topic, establishing some partial counterparts of the results proved in \cite{HaSk-morrey-comp}.

This paper is organised as follows. In Section 2 we present and collect some basic facts about smoothness Morrey spaces, on $\rn$ and on bounded domains $\Omega \subset \rn$, and introduce the notions of entropy and approximation numbers. In Section 3 we are concerned with the compactness of the above-described embeddings of Besov-type and Triebel-Lizorkin-type spaces on bounded domains. We also prove an extension of the results obtained in \cite{hs12b} for the scale $\MB$ to the cases when $p_i=u_i=\infty$, $i=1,2$. Moreover, we collect some immediate consequences of the main result, when we consider particular source and/or target spaces. In Section 4 we end up by characterising entropy and approximation numbers of the embedding $\id_\tau$ in \eqref{e0.0}, collecting also some special cases. 

}

\section{Preliminaries}
First we fix some notation. By $\nn$ we denote the \emph{set of natural numbers},
by $\nn_0$ the set $\nn \cup \{0\}$,  and by $\zn$ the \emph{set of all lattice points
in $\rn$ having integer components}.
For $a\in\real$, let   
$a_+:=\max\{a,0\}$.
All unimportant positive constants will be denoted by $C$, occasionally with
subscripts. By the notation $A \ls B$, we mean that there exists a positive constant $C$ such that
 $A \le C \,B$, whereas  the symbol $A \sim B$ stands for $A \ls B \ls A$.
We denote by $B(x,r) :=  \{y\in \rn: |x-y|<r\}$ the ball centred at $x\in\rn$ with radius $r>0$, and $|\cdot|$ denotes the Lebesgue measure when applied to measurable subsets of $\rn$.

Given two (quasi-)Banach spaces $X$ and $Y$, we write $X\hookrightarrow Y$
if $X\subset Y$ and the natural embedding of $X$ into $Y$ is continuous.

\subsection{Smoothness spaces of Morrey type on $\rn$}
Let $\SRn$ be the set of all \emph{Schwartz functions} on $\rn$, endowed
with the usual topology,
and denote by $\SpRn$ its \emph{topological dual}, namely,
the space of all bounded linear functionals on $\SRn$
endowed with the weak $\ast$-topology.
For all $f\in \cs(\rn)$ or $\cs'(\rn)$, we
use $\widehat{f}$ to denote its \emph{Fourier transform}, and $f^\vee$ for its inverse.
Let $\mathcal{Q}$ be the collection of all \emph{dyadic cubes} in $\rn$, namely,
$
\mathcal{Q}:= \{Q_{j,k}:= 2^{-j}([0,1)^d+k):\ j\in\zz,\ k\in\zn\}.
$
The {symbol}  $\ell(Q)$ denotes
the side-length of the cube $Q$ and $j_Q:=-\log_2\ell(Q)$.

Let $\vz_0,$ $\vz\in\SRn$ be such that
\begin{equation}\label{e1.0}
\supp \widehat{\vz_0}\subset \{\xi\in\rn:\,|\xi|\le2\}\, , \qquad
|\widehat{\vz_0}(\xi)|\ge C\ \text{if}\ |\xi|\le 5/3
\end{equation}
and
\begin{equation}\label{e1.1}
\supp \widehat{\vz}\subset \{\xi\in\rn: 1/2\le|\xi|\le2\}\quad\text{and}\quad
|\widehat{\vz}(\xi)|\ge C\ \text{if}\  3/5\le|\xi|\le 5/3,
\end{equation}
where $C$ is a positive constant.
In what follows, for all $\vz\in\cs(\rn)$ and $j\in\nn$, $\vz_j(\cdot):=2^{jd}\vz(2^j\cdot)$.

\begin{definition}\label{d1}
Let $s\in\rr$, $\tau\in[0,\infty)$, $q \in(0,\fz]$ and $\vz_0$, $\vz\in\cs(\rn)$
be as in \eqref{e1.0} and \eqref{e1.1}, respectively.
\bli
\item[{\bfseries\upshape (i)}]
Let $p\in(0,\infty]$. The \emph{Besov-type space} $\bt(\rn)$ is defined to be the collection of all $f\in\SpRn$ such that
$$\|f \mid {\bt(\rn)}\|:=
\sup_{P\in\mathcal{Q}}\frac1{|P|^{\tau}}\left\{\sum_{j=\max\{j_P,0\}}^\fz\!\!
2^{js q}\left[\int\limits_P
|\vz_j\ast f(x)|^p\dint x\right]^{\frac{q}{p}}\right\}^{\frac1q}<\fz$$
with the usual modifications made in case of $p=\fz$ and/or $q=\fz$.
\item[{\bfseries\upshape (ii)}]
Let $p\in(0,\infty)$. The \emph{Triebel-Lizorkin-type space} $\ft(\rn)$ is defined to be the collection of all $f\in \SpRn$ such that
$$\|f \mid {\ft(\rn)}\|:=
\sup_{P\in\mathcal{Q}}\frac1{|P|^{\tau}}\left\{\int\limits_P\left[\sum_{j=\max\{j_P,0\}}^\fz\!\!
2^{js q}
|\vz_j\ast f(x)|^q\right]^{\frac{p}{q}}\dint x\right\}^{\frac1p}<\fz$$
with the usual modification made in case of $q=\fz$.
\eli
\end{definition}

\begin{remark}\label{Rem-Ftau}
 These spaces were introduced in \cite{ysy}. To some extent the scale of Nikol'skij-Besov type spaces ${\bt(\rn)}$ had already been studied in \cite{ElBaraka1,ElBaraka2, ElBaraka3}.  
\end{remark}

We shall collect some features of these spaces below, but introduce first another scale of smoothness spaces of Morrey type. Recall first that the \emph{Morrey space}
  $\M(\rn)$, $0<p\le u<\infty $, is defined to be the set of all
  locally $p$-integrable functions $f\in L_p^{\mathrm{loc}}(\rn)$  such that
$$
\|f \mid {\M(\rn)}\| :=\, \sup_{x\in \rn, R>0} R^{\frac{d}{u}-\frac{d}{p}}
\left[\int_{B(x,R)} |f(y)|^p \dint y \right]^{\frac{1}{p}}\, <\, \infty\, .
$$

\begin{remark}
The spaces $\M(\rn)$ are quasi-Banach spaces (Banach spaces for $p \ge 1$).
They originated from Morrey's study on PDE (see \cite{Mor}) and are part of the wider class of Morrey-Campanato spaces; cf. \cite{Pee}. They can be considered as a complement to $L_p$ spaces. As a matter of fact, $\cM_{p,p}(\rn) = L_p(\rn)$ with $p\in(0,\infty)$.
To extend this relation, we put  $\cM_{\infty,\infty}(\rn)  = L_\infty(\rn)$. One can easily see that $\M(\rn)=\{0\}$ for $u<p$, and that for  $0<p_2 \le p_1 \le u < \infty$,
\begin{equation} \label{LinM}
	L_u(\rn)= \cM_{u,u}(\rn) \hookrightarrow  \cM_{u,p_1}(\rn)\hookrightarrow  \cM_{u,p_2}(\rn).
\end{equation}
In an analogous way, one can define the spaces $\cM_{\infty,p}(\rn)$, $p\in(0, \infty)$, but using the Lebesgue differentiation theorem, one can easily prove  that
$\cM_{\infty, p}(\rn) = L_\infty(\rn)$.
\end{remark}

Next we recall the definition of the other scale of smoothness spaces of Morrey type we deal with in this paper.

\begin{definition}\label{d2.5}
Let $0 <p\leq  u<\infty$ or $p=u=\infty$. Let  $q\in(0,\infty]$, $s\in \real$ and $\vz_0$, $\vz\in\cs(\rn)$
be as in \eqref{e1.0} and \eqref{e1.1}, respectively.
\bli
\item[{\upshape\bfseries (i)}]
The  {\em Besov-Morrey   space}
  $\MB(\rn)$ is defined to be the set of all distributions $f\in \SpRn$ such that
\begin{align}\label{BM}
\big\|f\mid \MB(\rn)\big\|:=
\bigg[\sum_{j=0}^{\infty}2^{jsq}\big\| \varphi_j \ast f\mid
\M(\rn)\big\|^q \bigg]^{1/q} < \infty
\end{align}
with the usual modification made in case of $q=\fz$.
\item[{\upshape\bfseries  (ii)}]
Let $u\in(0,\fz)$. The  {\em Triebel-Lizorkin-Morrey  space} $\MF(\rn)$
is defined to be the set of all distributions $f\in   \SpRn$ such that
\begin{align}\label{FM}
\big\|f \mid \MF(\rn)\big\|:=\bigg\|\bigg[\sum_{j=0}^{\infty}2^{jsq} |
 (\varphi_j\ast f)(\cdot)|^q\bigg]^{1/q}
\mid \M(\rn)\bigg\| <\infty
\end{align}
with the usual modification made in case of $q=\fz$.
\eli
\end{definition}

\begin{remark}
Besov-Morrey spaces were introduced by Kozono and Yamazaki in
\cite{KY}. They studied semi-linear heat equations and Navier-Stokes
equations with initial data belonging to  Besov-Morrey spaces.  The
investigations were continued by Mazzucato \cite{Maz}, where one can find the
atomic decomposition of the spaces. The Triebel-Lizorkin-Morrey spaces
were later introduced by  Tang and Xu \cite{TX}. We follow the
ideas of Tang and Xu \cite{TX}, where a somewhat  different definition is proposed. The ideas were further developed by Sawano and Tanaka \cite{ST1,ST2,Saw1,Saw2}. The most systematic and general approach to the spaces of this type  can  be found in the monograph \cite{ysy} or in the survey papers by Sickel \cite{s011,s011a}. 
\end{remark}

\begin{remark}
Note that for $u=p$ or $\tau=0$ we re-obtain the usual  Besov and Triebel-Lizorkin spaces:
\begin{equation}
  {\cal N}^{s}_{p,p,q}(\rn) = \B(\rn) = B^{s,0}_{p,q}(\rn) \label{MB=B}
\end{equation}
and
\begin{equation}
{\cal E}^{s}_{p,p,q}(\rn) = \F(\rn) = F^{s,0}_{p,q}(\rn), \label{MF=F}
\end{equation}
where $\B(\rn)$ and $\F(\rn)$ denote the classical
Besov spaces and Triebel-Lizorkin spaces,
respectively.  There exists extensive literature on such spaces; we
refer, in particular, to the series of monographs \cite{T-F1,T-F2,t06} for a
comprehensive treatment.
\end{remark}

\noindent{\em Convention.}~We adopt the nowadays usual custom to write $\A$ instead of $\B$ or $\F$, $\at$ instead of $\bt$ or $\ft$, and $\MA$ instead of $\MB$ or $\MF$, respectively, when both scales of spaces are meant simultaneously in some context.\\

We collect some basic properties of the scales $\at(\rn)$ and $\MA(\rn)$. 
The  spaces $\at(\rn)$ and $\MA(\rn)$ are
independent of the particular choices of $\vz_0$, $\vz$ appearing in their definitions. They are quasi-Banach spaces
(Banach spaces for $p,\,q\geq 1$), and $\mathcal{S}(\rn) \hookrightarrow
\MA(\rn), \at(\rn) \hookrightarrow \mathcal{S}'(\rn)$. In case of $\tau<0$ or $u<p$ we have $\at(\rn)=\MA(\rn)=\{0\}$.\\
 
Next we recall some basic embeddings results needed in the sequel. We refer to the references given above. For the spaces $\at(\rn)$ it is known that
\begin{equation} \label{elem-0-t}
A^{s+\ve,\tau}_{p,r}(\rn) \hookrightarrow \at(\rn) \qquad\text{if}\quad \varepsilon\in(0,\fz), \quad r,\,q\in(0,\infty],
\end{equation}
and
\begin{equation} \label{elem-1-t}
{A}^{s,\tau}_{p,q_1}(\rn)  \hookrightarrow {A}^{s,\tau}_{p,q_2}(\rn)\quad\text{if} \quad q_1\le q_2,
\end{equation}
 as well as
\begin{equation}\label{elem-tau}
B^{s,\tau}_{p,\min\{p,q\}}(\rn)\, \hookrightarrow \, \ft(\rn)\, \hookrightarrow \, B^{s,\tau}_{p,\max\{p,q\}}(\rn),
\end{equation}
which directly extends the well-known classical case from $\tau=0$ to $\tau\in [0,\infty)$, $p\in(0,\fz)$, $q\in(0,\fz]$ and $s\in\real$. 
Moreover, it is known from \cite[Proposition~2.6]{ysy}
that
\begin{equation} \label{010319}
A^{s,\tau}_{p,q}(\rn) \hookrightarrow B^{s+d(\tau-\frac1p)}_{\fz,\fz}(\rn).
\end{equation}
 The following  remarkable feature was proved in \cite{yy02}.
\begin{proposition}\label{yy02}
Let $s\in\rr$, $\tau\in[0,\fz)$  and $p,\,q\in(0,\fz]$ (with $p<\infty$ in the $F$-case). If
either $\tau>\frac1p$ or $\tau=\frac1p$ and $q=\infty$, then $A^{s,\tau}_{p,q}(\rn) = B^{s+d(\tau-\frac1p)}_{\fz,\fz}(\rn)$.
\end{proposition}

As for the scale $\MA(\rn)$ the counterparts to \eqref{elem-0-t}--\eqref{elem-tau} read as
\begin{equation} \label{elem-0}
{\mathcal A}^{s+\varepsilon}_{u,p,r}(\rn)  \hookrightarrow
\MA(\rn)\qquad\text{if}\quad \varepsilon>0, \quad r\in(0,\infty],
\end{equation}
and 
\begin{equation} \label{elem-1}
{\cal A}^{s}_{u,p,q_1}(\rn)  \hookrightarrow {\cal A}^{s}_{u,p,q_2}(\rn) \qquad\text{if}\quad q_1\le q_2.
\end{equation}
However, there also exist some differences.
Sawano proved in \cite{Saw2} that, for $s\in\real$ and $0<p< u<\infty$,
\begin{equation}\label{elem}
	{\cal N}^s_{u,p,\min\{p,q\}}(\rn)\, \hookrightarrow \, \MF(\rn)\, \hookrightarrow \,{\cal N}^s_{u,p,\infty}(\rn),
\end{equation}
where, for the latter embedding, $r=\infty$ cannot be improved -- unlike
in case of $u=p$ (see \eqref{elem-tau} with $\tau=0$). More precisely,
\[
\MF(\rn)\hookrightarrow {\mathcal N}^s_{u,p,r}(\rn)\quad\text{if, and only if,}\quad r=\infty ~~ \text{or} ~~  u=p\ \text{and}\ r\ge \max\{p,\,q\}.
\]
On the other hand, Mazzucato has shown in \cite[Proposition~4.1]{Maz} that
\[
\mathcal{E}^0_{u,p,2}(\rn)=\M(\rn),\quad 1<p\leq u<\infty,
\]
in particular,
\begin{equation}\label{E-Lp}
\mathcal{E}^0_{p,p,2}(\rn)=L_p(\rn)=F^0_{p,2}(\rn),\quad p\in(1,\infty).
\end{equation}

\begin{remark}\label{N-Bt-spaces}
We obtained a lot more embedding results within the scales of spaces $\MA(\rn)$ and $\at(\rn)$, respectively, in \cite{hs12b,hs14,YHSY,YHMSY}, but will recall some of them in detail below as far as needed for our argument. We turn to the relation between the two scales of smoothness Morrey spaces. 
Let $s$, $u$, $p$ and $q$ be as in Definition~\ref{d2.5}
and $\tau\in[0,\fz)$.
It is known from \cite[Corollary~3.3, p. 64]{ysy} that
\begin{equation}
\MB(\rn) \hookrightarrow  \bt(\rn) \qquad \text{with}\qquad \tau=\frac{1}{p}- \frac{1}{u}.
\label{N-BT-emb}
\end{equation}
Moreover, the above embedding is proper if $\tau>0$ and $q<\infty$. If $\tau=0$ or $q=\infty$, then both spaces coincide with each other, in particular,
\begin{equation}
\mathcal{N}^{s}_{u,p,\infty}(\rn)  =  B^{s,\frac{1}{p}- \frac{1}{u}}_{p,\infty}(\rn).
\label{N-BT-equal}
\end{equation}
As for the $F$-spaces, if $0\le \tau <{1}/{p}$,
then
\begin{equation}\label{fte}
\ft(\rn)\, = \, \MF(\rn)\quad\text{with }\quad \tau =\frac{1}{p}-\frac{1}{u}\, ,\quad 0 < p\le u < \infty\, ;
\end{equation}
cf. \cite[Corollary~3.3, p.\,63]{ysy}. Moreover, if $p\in(0,\infty)$ and $q\in(0,\infty]$, then
\begin{equation}\label{ftbt}
F^{s,\, \frac{1}{p} }_{p\, ,\,q}(\rn) \, = \, F^{s}_{\infty,\,q}(\rn)\, = \, B^{s,\, \frac1q }_{q\, ,\,q}(\rn) \, ;
\end{equation}
cf. \cite[Propositions~3.4 and 3.5]{s011} and \cite[Remark 10]{s011a}.
\end{remark}

For later use we recall the definition of the space $\bmo(\rn)$, i.e., the local (non-homogeneous) space of functions of bounded mean oscillation, consisting of all locally integrable
functions $\ f\in \Lloc(\rn) $ satisfying that
\begin{equation*}
 \left\| f \right\|_{\bmo}:=
\sup_{|Q|\leq 1}\; \frac{1}{|Q|} \int\limits_Q |f(x)-f_Q| \dint x + \sup_{|Q|>
1}\; \frac{1}{|Q|} \int\limits_Q |f(x)| \dint x<\infty,
\end{equation*}
where $ Q $ appearing in the above definition runs over all cubes in $\rn$, and $ f_Q $ denotes the mean value of $ f $ with
respect to $ Q$, namely, $ f_Q := \frac{1}{|Q|} \;\int_Q f(x)\dint x$,
cf. \cite[2.2.2(viii)]{T-F1}. Hence the above result \eqref{ftbt} implies, in particular,
\begin{equation}\label{ft=bmo}
\bmo(\rn)= F^{0}_{\infty,2}(\rn)= F^{0, 1/p}_{p, 2}(\rn), \quad 0<p<\infty.
\end{equation}

\begin{remark}\label{T-hybrid}
In contrast to this approach, Triebel followed the original Morrey-Campanato ideas to develop local spaces $\mathcal{L}^r\A(\rn)$ in \cite{t13}, and so-called `hybrid' spaces $L^r\A(\rn)$ in \cite{t14}, where $0<p<\infty$, $0<q\leq\infty$, $s\in\real$, and $-\frac{d}{p}\leq r<\infty$. This construction is based on wavelet decompositions and also combines local and global elements as in Definitions~\ref{d1} and \ref{d2.5}. However, Triebel proved in \cite[Chapter~3]{t14} that
\begin{equation} \label{hybrid=tau}
L^r\A(\rn) = \at(\rn), \qquad \tau=\frac1p+\frac{r}{d},
\end{equation}
in all admitted cases. Therefore we do not have to deal with these spaces separately in the sequel.
\end{remark}

\ignore{\begin{proposition}\label{MorreyintoC}
\bli
\item[{\hfill\bfseries\upshape (i)\hfill}]
Let $s\in \rr$, $\tau\in(0,\fz)$ and $p,\,q\in(0,\fz]$ (with $p<\infty$ in the case of $F$-spaces). The embedding
$$
\at(\rn)\hookrightarrow L_{\infty}(\rn)
$$
holds if, and only if,
$$
s>d\left(\frac{1}{p}-\tau\right).
$$
\item[{\hfill\bfseries\upshape (ii)\hfill}]
 Let $s\in\rr$, $0<p < u <\infty$ and $q\in(0,\fz]$. The embedding
\begin{equation*}
\MB(\rn) \hookrightarrow L_{\infty}(\rn)\,
\end{equation*}
holds  if, and only if,
\begin{equation*}
\begin{cases}
q\in(0,\infty], & \text{if}\quad s>\frac{d}{u}, \\ q\in(0,1], & \text{if} \quad s=\frac{d}{u}.
\end{cases}
\end{equation*}
\item[{\hfill\bfseries\upshape (iii)\hfill}]
 Let $s\in\rr$, $0<p < u <\infty$ and $q\in(0,\fz]$. The embedding
\begin{equation*}
\MF(\rn) \hookrightarrow L_{\infty}(\rn)\,
\end{equation*}
holds  if, and only if,
\begin{equation*}
s> \frac{d}{u}.
\end{equation*}
\eli
\end{proposition}

Part (i) of the above proposition coincides with \cite[Proposition 4.1]{YHMSY} and covers \cite[Proposition 5.4, Corollary 5.5]{YHSY}, whereas parts (ii) and (iii) can be found in \cite[Proposition 5.5]{hs12b} and  \cite[Proposition 3.8]{hs14}.\\
}

\subsection{Spaces on domains}

We assume that  $\Omega$ is a bounded $C^{\infty}$ domain in $\rn$.  We consider smoothness Morrey spaces on $\Omega$ defined by restriction. Let ${\cal D}(\Omega)$ be the set of all infinitely differentiable functions supported in $\Omega$ and denote by ${\cal D}'(\Omega)$ its dual.
  Since we are able to define the extension operator $\ext: {\cal D}(\Omega) \rightarrow  \SRn$, cf.
\cite{Saw2010}, the restriction operator $\re :   \SpRn \rightarrow {\cal D}'(\Omega)$ can be defined naturally as an adjoint operator
\[
\langle \re  (f), \varphi\rangle= \langle f, \ext (\varphi)\rangle, \quad f\in\SpRn,
\]
where $\varphi\in \mathcal{D}(\Omega)$. We will write $f\vert_{\Omega}={\rm re } (f)$.


\begin{definition}\label{D-spaces-Omega}
Let $s\in\real$ and $q\in (0,\infty]$.
\bli
\item[{\hfill\bfseries\upshape (i)\hfill}]
Let $0 <p\leq  u<\infty$ or $p=u=\infty$ (with $u<\infty$ in case of ${\cal A}={\cal E}$). Then
$\MA(\Omega)$ is defined by
\[
\MA(\Omega):=\big\{f\in {\cal D}'(\Omega): f=g\vert_{\Omega} \text{ for some } g\in \MA(\rn)\big\}
\]
endowed with the quasi-norm
\[
\big\|f\mid \MA(\Omega)\big\|:= \inf \big\{ \|g\mid \MA(\rn)\|:  f=g\vert_{\Omega}, \; g\in  \MA(\rn)\big\}.
\]
\item[{\hfill\bfseries\upshape (ii)\hfill}]
Let $\tau\in [0,\infty)$ and $p\in (0,\infty]$ (with $p<\infty$ in case of $\at=\ft$). Then
$\at(\Omega)$ is defined by
\[
\at(\Omega):=\big\{f\in {\cal D}'(\Omega): f=g\vert_{\Omega} \text{ for some } g\in \at(\rn)\big\}
\]
endowed with the quasi-norm
\[
\big\|f\mid \at(\Omega)\big\|:= \inf \big\{ \|g\mid \at(\rn)\|:  f=g\vert_{\Omega}, \; g\in  \at(\rn)\big\}.
\]
\eli
\end{definition}

\begin{remark}\label{tau-onOmega}
The spaces $\MA(\Omega)$ and $\MA(\Omega)$ are quasi-Banach spaces (Banach spaces for $p,q\geq 1$).  When $u=p$ or $\tau=0$ we re-obtain
the usual Besov and Triebel-Lizorkin spaces defined on bounded smooth domains.
Several properties of the spaces $\MA(\Omega)$, including  the extension property, were studied in \cite{Saw2010}. As for the spaces $\at(\Omega)$ we also refer to \cite[Section~6.4.2]{ysy}. In particular, according to \cite[Theorem~6.13]{ysy}, there exists a linear and bounded extension operator
\begin{equation}\label{ext-tau-1}
\ext_\tau: \at(\Omega)\to \at(\rn),\quad \text{where}\quad 1\leq p<\infty, 0<q\leq\infty, s\in\real, \tau\geq 0,
\end{equation}
such that
\begin{equation}\label{ext-tau-2}
\mathrm{re} \circ \ext_\tau  = \id\quad\text{in}\quad \at(\Omega),
\end{equation}
where $\mathrm{re}: \at(\rn)\to \at(\Omega)$ is the restriction operator as above.

Embeddings within the scale of spaces $\MA(\Omega)$ as well as to classical spaces like $C(\Omega)$ or $L_r(\Omega)$ were investigated in \cite{hs12b,hs14}.
In \cite{hms} we studied the question under what assumptions these spaces consist of regular distributions only. In \cite{YHMSY} we considered the approximation numbers of some special compact embedding of $\at(\Omega)$ into $L_\infty(\Omega)$.

\end{remark}

\ignore{\begin{remark}

\end{remark}
}


\ignore{\begin{definition}\label{tau-spaces-Omega}
Let $s\in\real$, $\tau\in [0,\infty)$, $q\in (0,\infty]$ and $p\in (0,\infty]$ (with $p<\infty$ in case of $\at=\ft$). Then
$\at(\Omega)$ is defined by
\[
\at(\Omega):=\big\{f\in {\cal D}'(\Omega): f=g\vert_{\Omega} \text{ for some } g\in \at(\rn)\big\}
\]
endowed with the quasi-norm
\[
\big\|f\mid \at(\Omega)\big\|:= \inf \big\{ \|g\mid \at(\rn)\|:  f=g\vert_{\Omega}, \; g\in  \at(\rn)\big\}.
\]
\end{definition}
}

\begin{remark}\label{emb-Omega}
  Let us mention that we have the counterparts of many continuous embeddings stated in the previous subsection for spaces on $\rn$ when dealing with spaces restricted to bounded domains. This concerns, in particular, the elementary embeddings and coincidences \eqref{elem-0-t}--\eqref{elem-tau}, Proposition~\ref{yy02} and \eqref{elem-0}--\eqref{fte}. 
\end{remark}

\subsection{Entropy numbers}

\label{sect-ek}

As explained in the beginning already, our main concern in this paper is to characterise the compactness of embeddings in further detail. Therefore we briefly recall the concepts of entropy and approximation numbers.


\begin{definition}\label{defi-ak}
Let $ X $ and $Y$ be two complex (quasi-) Banach spaces,
$k\in\nn\ $ and let $\ T\in\mathcal{L}(X,Y)$ be a linear and
continuous operator from $ X $ into $Y$. 
\bli
\item[{\hfill\bfseries\upshape (i)\hfill}]
The {\em k\,th entropy number} $\ e_k(T)\ $ of $\ T\ $ is the
infimum of all numbers $\ \varepsilon>0\ $ such that there exist $\ 2^{k-1}\ $ balls
in $\ Y\ $ of radius $\ \varepsilon\ $ which cover the image $\ T\,B_X$ of the unit ball $\ B_X=\{x\in
X:\;\|x|X\|\leq 1\}$.
\item[{\hfill\bfseries\upshape (ii)\hfill}]
The {\em k\,th approximation number} $\ a_k(T)\ $ of $\ T\ $ is defined by
\begin{equation}
  a_k(T) = \inf \{ \| T - S \| : S \in {\mathcal L}(X,Y), \,
  \mathrm{rank}\, S < k \},\quad k\in\nn.
  \label{approx-n}
  \end{equation}
\eli
\end{definition}

\begin{remark}
For details and properties of entropy and approximation numbers we refer to \cite{CS,EE,Koe,Pie-s} (restricted to the case of Banach spaces), and \cite{ET} for some extensions to quasi-Banach spaces. Among other
features we only want to mention
the multiplicativity of entropy numbers: let $X,Y,Z$
be complex (quasi-) Banach spaces and $\ T_1 \in\mathcal{L}(X,Y)$, $ T_2 \in\mathcal{L}(Y,Z)$. Then
\beq
e_{k_1+k_2-1} (T_2\circ T_1) \leq e_{k_1}(T_1)\,
e_{k_2} (T_2),\quad k_1, k_2\in\nn.
\label{e-multi}
\eeq
Note that one has in general
$\ \lim_{k\rightarrow\infty} e_k(T)= 0\ $ {if, and only if,}
$\ T\ $ {is compact}. 
The last equivalence justifies the saying that entropy numbers measure
`how compact' an operator acts. This is one reason to study the asymptotic
behaviour of entropy numbers (that is, their decay) for compact operators in
detail.

Approximation numbers share many of the basic features of entropy numbers, but are different in some respect. They can -- unlike entropy numbers -- be regarded as special {\itshape
  $s$-numbers}, a concept introduced by {Pietsch} \cite[Secttion~11]{Pia}. Of special importance is the close connection of both concepts, entropy numbers as well as approximation numbers, with spectral theory, in particular, the estimate of eigenvalues. We refer to the monographs \cite{CS,EE,ET,Koe,Pie-s} for further details.
\end{remark}

\begin{remark}\label{remark-ek-ak-class}
We recall what is well-known in the case of the embedding
\[
\id_A:
\Ae(\Omega) \to \Az(\Omega),
\]
where $\ -\infty<s_2\leq s_1<\infty$, $ 0<p_1, p_2\leq\infty$
($p_1, p_2<\infty $ in the $F$-case), $ 0<q_1, q_2\leq\infty$, and the spaces $\A(\Omega)$ are defined by restriction. Let
\begin{equation}
\delta = s_1-s_2-d\left(\frac{1}{p_1}-\frac{1}{p_2}\right), \quad \delta_+= s_1-s_2-d\left(\frac{1}{p_1}-\frac{1}{p_2}\right)_+ .
\label{delta+}
\end{equation}
Then $\id_A $ is compact when $\ \delta_+>0$; cf. \cite[(2.5.1/10)]{ET}.
In this situation  {Edmunds} and {Triebel} proved in \cite{ET1,ET2} (see also \cite[Theorem~3.3.3/2]{ET}) that
\begin{equation}
e_k(\id_A) \ \sim \ k^{-\frac{s_1-s_2}{d}},\quad k\in\nn,\label{ek-id_A}
\end{equation}
where $ s_1\geq s_2$, $ 0<p_1, p_2\leq\infty\ (p_1, p_2<\infty $ in the
$ F$-case), $ 0<q_1, q_2\leq\infty$, and $\ \delta_+>0$. In the case of
approximation numbers the situation is more complicated; the result of
{Edmunds} and {Triebel} in \cite[Theorem~3.3.4]{ET}, partly
improved by {Caetano} \cite{Cae}, 
reads as
\begin{equation}
a_k(\id_A)  \quad\sim\quad k^{-\frac{\delta_+}{d} - \varkappa},\quad k\in\nn,\label{ak-id_A}
\end{equation}
with
\begin{equation}
\varkappa\ =   \left(\frac{\min\{p_1',p_2\}}{2}-1\right)_+ \cdot\min
\left\{\frac{\delta}{d},\frac{1}{\min\{p_1',p_2\}}\right\},
\end{equation}
where $\delta$ is given by \eqref{delta+} and $p_1'$ denotes the conjugate of $p_1$ defined by $\frac{1}{p_1} + \frac{1}{p_1'}=1$ if $1\leq p_1 \leq \infty$ and $p_1'=\infty$ if $0<p_1 <1$. The above asymptotic result is almost complete
now, apart from the restrictions that $(p_1, p_2) \neq (1,\infty)$  or $\frac{\delta}{d}\not= \frac{1}{\min\{p'_1, p_2\}}$ when
$0<p_1< 2 <p_2 \leq\infty$. Note that $\varkappa=0$ unless $p_1<2<p_2$, and $\delta\geq \delta_+$ with $\delta=\delta_+$ if $p_1\leq p_2$.
\end{remark}

\section{Compact embeddings}\label{sect-comp}

First we recall our compactness result as obtained in \cite{hs12b} (for $\mathcal{A}=\mathcal{N}$) and \cite{hs14} (for $\mathcal{A}=\mathcal{E}$). We shall heavily rely on this result in our argument below.\\

\noindent{\em Convention.}
 Here and in the sequel we shall understand $\frac{p_i}{u_i}=1$ in case of $p_i=u_i=\infty$, $i=1,2$.

\begin{theorem}  \label{comp}
  Let  $s_i\in \real$, $0<q_i\leq\infty$, $0<p_i\leq u_i<\infty$, {or $p_i=u_i=\infty$ in the case of $\mathcal N$-spaces},  $i=1,2$.
Then  the embedding
\begin{equation} \label{bd1comp}
	 \id_{\mathcal{A}}: \MAe(\Omega )\hookrightarrow \MAz(\Omega )
\end{equation}
is compact if, and only if,
the following condition holds
\begin{equation}\label{bd3acomp}
\frac{s_1-s_2}{d} > \max\bigg\{0,\frac{1}{u_1} - \frac{1}{u_2},
\frac{p_1}{u_1} \Big(\frac{1}{p_1}- \frac{1}{p_2}\Big) \bigg\}.
\end{equation}

In  particular, if $p_1=u_1=\infty$ and $\MAe=\MBe$, then $\id_{\mathcal{A}}$ given by \eqref{bd1comp} is compact if, and only if,  $s_1>s_2$.
If $p_2=u_2=\infty$  and $\MAz=\MBz$, then $\id_{\mathcal{A}}$ is compact if, and only if,
$\frac{s_1-s_2}{d} > \frac{1}{u_1}$.
\end{theorem}

\begin{proof}
The cases when $0<p_i\leq u_i<\infty$, $i=1,2$, were proved in \cite{hs12b} and \cite{hs14} for $\mathcal{A}=\mathcal{N}$ and $\mathcal{A}=\mathcal{E}$ respectively. So we are left with the cases $p_1=u_1=\infty$ or $p_2=u_2=\infty$.

At first, let us consider the case when $p_1=u_1=\infty$ and $\MAe=\MBe$. If $s_1-s_2>0$, the compactness of $\id_{\mathcal{A}}$ follows from
\begin{equation}\label{pu-1-infty-1}
\mathcal{N}^{s_1}_{\infty,\infty,q_1}(\Omega)= B^{s_1}_{\infty, q_1} (\Omega)\hookrightarrow A^{s_2}_{u_2,q_2}(\Omega) \hookrightarrow \mathcal{A}^{s_2}_{u_2,p_2,q_2}(\Omega),
\end{equation}
as the first embedding is compact when $s_1-s_2>0$.

Now we assume that $\id_{\mathcal{A}}$ is compact. We have
 \begin{equation}\label{pu-1-infty-2}
B^{s_1}_{\infty, q_1} (\Omega)= \mathcal{N}^{s_1}_{\infty,\infty,q_1}(\Omega) \hookrightarrow \mathcal{A}^{s_2}_{u_2,p_2,q_2}(\Omega) \hookrightarrow A^{s_2}_{p_2,q_2}(\Omega),
\end{equation}
where the last embedding was proved in \cite{hs12b} and \cite{hs14}. Then, the compactness of the first embedding implies the compactness of the embedding between the outer spaces, which in turn implies $s_1-s_2>0$.


    {Now let $p_2=u_2=\infty$, $\MAz=\MBz$ and $\frac{s_1-s_2}{d} > \frac{1}{u_1}$. As the case $p_1=u_1=\infty$ (when $\mathcal{A}=\mathcal{N}$) is already covered by our preceding observation, we may further assume that $0<p_1\leq u_1<\infty$. By a  straightforward extension of our continuity result in \cite[Theorem~3.1]{hs12b} (to the cases when $p_i=u_i=\infty$ for $i=1$ or $i=2$)  we have the continuous embedding
\begin{equation} \label{pu-2-infty-1}
\MBe(\Omega) \hookrightarrow B^{s_1-\frac{d}{u_1}}_{\infty,\infty}(\Omega). 
\end{equation}
Moreover, in case of $s_1-\frac{d}{u_1}>s_2$, it is well-known that the embedding
\begin{equation}\label{pu-2-infty-2}
  B^{s_1-\frac{d}{u_1}}_{\infty,\infty}(\Omega) \hookrightarrow  B^{s_2}_{\infty, q_2} (\Omega) = \mathcal{N}^{s_2}_{\infty,\infty,q_2}(\Omega)
\end{equation}
is compact.
  Thus $\MBe(\Omega) \hookrightarrow \mathcal{N}^{s_2}_{\infty,\infty,q_2}(\Omega)$ compactly for any $q_1, q_2\in (0,\infty]$. The compactness of $\MFe(\Omega) \hookrightarrow \mathcal{N}^{s_2}_{\infty,\infty,q_2}(\Omega)$ is then a consequence of \eqref{elem}.
  }

  The necessity follows from the following chain of embeddings
\begin{equation}\label{pu-2-infty-3} 
B^{s_1}_{u_1, q_1} (\Omega) \hookrightarrow \MBe (\Omega) \hookrightarrow  \mathcal{N}^{s_2}_{\infty,\infty,q_2}(\Omega)=B^{s_2}_{\infty, q_2} (\Omega)
\end{equation}
in the same way as above. Finally we apply \eqref{elem} for the case $\MAe=\MFe$.
\end{proof}




Now we give the counterpart of Theorem~\ref{comp} for Besov-type and Triebel-Lizorkin-type spaces. For convenience we use some abbreviation for the following expression, which plays an essential role in the sequel. So let us denote

\begin{equation}
\critical :=
\max\left\{\left(\tau_2-\frac{1}{p_2}\right)_+ -\left(\tau_1-\frac{1}{p_1}\right)_+,                                  \frac{1}{p_1}-\tau_1 - \min\left\{\frac{1}{p_2}-\tau_2, \frac{1}{p_2}(1-p_1\tau_1)_+\right\}\right\}.
 \label{gamma}
\end{equation}


\begin{theorem} \label{comp-tau-u}
  Let  $s_i\in \real$, $0<q_i\leq\infty$, $0<p_i\leq \infty$ (with $p_i<\infty$ in case of $A=F$), $\tau_i\geq 0$, $i=1,2$.
    The embedding
\begin{equation} \label{tau-comp-u1}
	 \id_{\tau}: \ate(\Omega )\hookrightarrow \atz(\Omega )
\end{equation}
is compact if, and only if,
the following condition holds
\begin{equation}\label{tau-comp-u2}
  \frac{s_1-s_2}{d} > \critical.
\end{equation}
\end{theorem}

\begin{proof}
We shall use sharp embeddings and identities like  \eqref{010319}, Proposition \ref{yy02} and  \eqref{fte} (all adapted to spaces restricted to the domain $\Omega$, recall Remark~\ref{emb-Omega})
    together with our previous result Theorem~\ref{comp} several times. Therefore we shall always  distinguish below between the cases $\tau_i<\frac{1}{p_i}$ and $\tau_i\geq \frac{1}{p_i}$ (with some additional restrictions on $q_i$ occasionally), $i=1,2$. For that reason it seems convenient to reformulate condition \eqref{tau-comp-u2} according to these cases; that is, the goal is to prove that $\id_\tau$ given by \eqref{tau-comp-u1} is compact if, and only if,
  \begin{equation}\label{reform}
    \frac{s_1-s_2}{d} > \begin{cases}
      \frac{1}{p_1}-\tau_1-\frac{1}{p_2}+\tau_2, & \text{if}\quad \tau_2\geq \frac{1}{p_2}, \\[1ex]
\frac{1}{p_1}-\tau_1, &\text{if}\quad  \tau_1\geq \frac{1}{p_1}, \ \tau_2< \frac{1}{p_2}, \\[1ex]
     \max\left\{0, \frac{1}{p_1}-\tau_1-\frac{1}{p_2}+\max\left\{\tau_2,\frac{p_1}{p_2}\tau_1\right\}\right\}, &\text{if}\quad  \tau_1< \frac{1}{p_1}, \  \tau_2< \frac{1}{p_2}.
    \end{cases}
  \end{equation}

  \emph{Step 1.} Let us assume that $\tau_2\geq \frac{1}{p_2}$ with $q_2=\infty$ {if} $\tau_2=\frac{1}{p_2}$. First we prove the sufficiency of \eqref{reform} for the compactness of $\id_\tau$, that is, we assume now
  \begin{equation}\label{reform-1}
    s_1+d\left(\tau_1-\frac{1}{p_1}\right)>s_2+d\left(\tau_2-\frac{1}{p_2}\right).
  \end{equation}
  Then the  embedding \eqref{010319} and Proposition \ref{yy02} yield
\begin{align}\label{tlarge1}
\ate(\Omega )\hookrightarrow  B^{s_1+d(\tau_1-\frac{1}{p_1})}_{\infty,\infty}(\Omega)\hookrightarrow  B^{s_2+d(\tau_2-\frac{1}{p_2})}_{\infty,\infty}(\Omega) = \atz(\Omega )
\end{align}
and the embedding between the Besov spaces is compact. So $\id_\tau$ is compact.

Now we turn to the necessity of \eqref{reform-1} for the compactness of $\id_\tau$. So we assume that $\id_\tau$ given by \eqref{tau-comp-u1} is compact. Let first $\tau_1\geq \frac{1}{p_1}$ with $q_1=\infty$ if $\tau_1=\frac{1}{p_1}$. Then by Proposition~\ref{yy02} we obtain
\begin{align}\label{tlarge2}
B^{s_1+d(\tau_1-\frac{1}{p_1})}_{\infty,\infty}(\Omega)= \ate(\Omega ) \hookrightarrow \atz(\Omega )=  B^{s_2+d(\tau_2-\frac{1}{p_2})}_{\infty,\infty}(\Omega)
\end{align}
which results in a compact embedding between the outer Besov spaces. This is well-known to imply \eqref{reform-1} as desired.  A similar argument works for $p_1=\infty$, $\tau_1=0$ and $q_1<\infty$, since then we have
 \begin{align*}
B^{s_1}_{\infty,q_1}(\Omega)= \bte(\Omega ) \hookrightarrow \atz(\Omega )=  B^{s_2+d(\tau_2-\frac{1}{p_2})}_{\infty,\infty}(\Omega).
\end{align*}

 Let now  $0<\tau_1 \leq \frac{1}{p_1}$ with $q_1<\infty$ if $\tau_1=\frac{1}{p_1}$. We take {$p_0\in (0,\infty)$ such that $\frac{1}{p_0}>\frac{1}{p_1}-\tau_1$ }
 and $q_0=\min\{p_1,q_1\}$. Then {Corollary 5.2} in \cite{YHSY} implies
\begin{align}\label{below-tau-small}
B^{s_1+d(\tau_1-\frac{1}{p_1}+\frac{1}{p_0})}_{p_0,q_0}(\Omega)\hookrightarrow B^{s_1,\tau_1}_{p_1,q_0}(\Omega) \hookrightarrow\ate(\Omega ) \hookrightarrow \atz(\Omega )=  B^{s_2+d(\tau_2-\frac{1}{p_2})}_{\infty,\infty}(\Omega).
\end{align}
Once more the compactness of  $\id_\tau$ implies \eqref{reform-1}.  If $\tau_1=0$, then $B^{s_1}_{p_1,q_0}(\Omega)$ can be embedded into $\ate(\Omega)$ and the similar argument holds. This concludes the proof in that case of $\tau_2\geq \frac{1}{p_2}$.\\

\emph{Step 2}. Next we assume that $0\leq \tau_2 < \frac{1}{p_2}$ and benefit from the coincidence \eqref{fte}.\\

{\em Substep 2.1}. 
{If also $0\leq \tau_1 < \frac{1}{p_1}$, then \eqref{tau-comp-u2} reads as
\begin{equation}\label{reform-2}
    \frac{s_1-s_2}{d}>\max\left\{0, \frac{1}{p_1}-\tau_1-\frac{1}{p_2}+\max\left\{\tau_2,\frac{p_1}{p_2}\tau_1\right\}\right\}.
\end{equation}
For the $F$-spaces, the result immediately follows from coincidence \eqref{fte} and Theorem \ref{comp}. Note that in this case \eqref{bd3acomp} coincides with the last line in \eqref{reform} in view of $\frac{1}{u_i}=\frac{1}{p_i}-\tau_i$, $i=1,2$, as required by \eqref{fte}.

We shall then prove the result for the Besov scale. At first, let us assume \eqref{reform-2} holds. Then \eqref{elem-1-t}, \eqref{N-BT-emb} and \eqref{N-BT-equal} yield
\begin{align}\label{tau2small-1}
\bte(\Omega)\hookrightarrow  B^{s_1, \tau_1}_{p_1, \infty}(\Omega)={\cal N}^{s_1}_{u_1, p_1, \infty}(\Omega)\hookrightarrow \MBz(\Omega) \hookrightarrow \btz(\Omega)
\end{align}
and the embedding between the Besov-Morrey spaces is compact. Therefore $\id_\tau$ is compact. 

We now turn to the necessity and assume that $\id_\tau$ is compact. Then
\begin{align}\label{tau2small-2}
\MBe(\Omega) \hookrightarrow  \bte (\Omega) \hookrightarrow \btz(\Omega) \hookrightarrow B^{s_2, \tau_2}_{p_2, \infty}(\Omega) = {\cal N}^{s_2}_{u_2, p_2, \infty}(\Omega),
\end{align}
where we have used again \eqref{N-BT-emb}, \eqref{elem-1-t} and the coincidence \eqref{N-BT-equal}. In view of Theorem \ref{comp} this leads to the desired condition \eqref{reform-2}. \\}

{\em Substep 2.2}. An analogous argument works for $\tau_1 \ge \frac{1}{p_1}$ with $q_1=\infty$ if $\tau_1 = \frac{1}{p_1}$, cf. Proposition \ref{yy02}. This time  \eqref{tau-comp-u2} coincides with the second line of \eqref{reform}. 
{So for the sufficiency we assume that 
\begin{equation}\label{reform-3}
    \frac{s_1-s_2}{d}>\tau_1-\frac{1}{p_1}.
\end{equation}
Then Proposition \ref{yy02}, the embedding \eqref{N-BT-emb} and the coincidence \eqref{fte} give
\begin{align} \label{tau2small-3}
\ate(\Omega) = B^{s_1+d(\tau_1-\frac{1}{p_1})}_{\infty,\infty} (\Omega) = {\cal N}^{s_1+d(\tau_1-\frac{1}{p_1})}_{\infty, \infty, \infty}(\Omega) \hookrightarrow \MAz (\Omega) \hookrightarrow \atz (\Omega).
\end{align}
Then, by Theorem \ref{comp}, $\id_\tau$ is compact. 
Conversely, let us assume that $\id_\tau$ is compact. We benefit from Proposition \ref{yy02} and the coincidences \eqref{N-BT-equal} and \eqref{fte} to obtain
\begin{align} \label{tau2small-4}
{\cal N}^{s_1+d(\tau_1-\frac{1}{p_1})}_{\infty, \infty, \infty}(\Omega) = B^{s_1+d(\tau_1-\frac{1}{p_1})}_{\infty, \infty}(\Omega)= \ate(\Omega) \hookrightarrow \atz (\Omega) \hookrightarrow {\cal N}^{s_2}_{u_2, p_2,\infty}(\Omega).
\end{align}
In case of $A=F$, the last embedding is true due to the elementary embeddings \eqref{elem}. Therefore, the compactness of $\id_\tau$ and Theorem \ref{comp} lead to the desired condition $\frac{s_1-s_2}{d}>\tau_1-\frac{1}{p_1}$.
}
\\

{\em Substep 2.3}. Assume finally $\tau_1=\frac{1}{p_1}$ with $q_1<\infty$. Note that in this case, due to the middle line of \eqref{reform}, the condition \eqref{tau-comp-u2} reads as $s_1>s_2$. We apply \eqref{010319} to obtain
\begin{equation}\label{lim-1}
\ate(\Omega)=A^{s_1,\frac{1}{p_1}}_{p_1,q_1}(\Omega)\hookrightarrow  B^{s_1}_{\infty,\infty}(\Omega)\hookrightarrow \MAz(\Omega) \hookrightarrow \atz(\Omega)  \quad\text{with}\quad  \frac{1}{u_2}=\frac{1}{p_2}-\tau_2>0.
\end{equation}
The second embedding is compact for $s_1>s_2$ in view of Theorem \ref{comp} and the last embedding is a consequence of \eqref{N-BT-emb} and \eqref{fte}, respectively. Conversely, if $\id_\tau$ is compact in this case, then we can argue as follows. Put $\frac{1}{u_2}=\frac{1}{p_2}-\tau_2$. Then \cite[Corollary~5.2]{YHSY} (for $A=B$) and \eqref{ftbt} (for $A=F$) lead to
\begin{equation}\label{lim-2}
B^{s_1}_{\infty,q_1}(\Omega)\hookrightarrow \ate(\Omega) \hookrightarrow \atz(\Omega) \hookrightarrow A^{s_2,\tau_2}_{p_2,v_2}(\Omega) = \mathcal{A}^{s_2}_{u_2,p_2,v_2}(\Omega)\hookrightarrow \mathcal{N}^{s_2}_{u_2,p_2,\infty}(\Omega)
\end{equation}
with $v_2=\infty$ if $A=B$, and $v_2\geq q_2$ if $A=F$, where we used \eqref{N-BT-equal} and \eqref{fte} in the last equality and \eqref{elem} in the last embedding. In view of Theorem~\ref{comp} this leads to $s_1>s_2$ as required.\\

\emph{Step 3}. It remains to deal with $\tau_2 = \frac{1}{p_2}$ and $q_2<\infty$. In that case \eqref{reform} always reads as $s_1-s_2>d(\frac{1}{p_1}-\tau_1)$.\\

{\em Substep 3.1}. Assume first $\tau_1 < \frac{1}{p_1}$ and let $\frac{1}{u_1}=\frac{1}{p_1}-\tau_1$. {Then by elementary embeddings and the coincidences \eqref{N-BT-equal} and \eqref{fte}, 
\begin{align}\label{lim-6}
\ate(\Omega) \hookrightarrow \mathcal{N}^{s_1}_{u_1,p_1,\infty}(\Omega) \hookrightarrow \mathcal{N}^{s_2}_{\infty, \infty,q_0}(\Omega) = B^{s_2}_{\infty,q_0}(\Omega) \hookrightarrow B^{s_2,\tau_2}_{p_2,q_0}(\Omega), 
\end{align}
and the embedding of the outer spaces is compact for any $q_0$, since the second embedding is compact by Theorem \ref{comp} with \eqref{tau-comp-u2}. The last embedding is continuous where we apply \cite[Theorem~2.5]{YHSY}. If $A=B$, we put $q_0=q_2$ and the argument is complete, while in case of $A=F$ we choose $q_0\leq \min\{p_2,q_2\}$ and finally use the continuous embedding into $\ftz(\Omega)$ due to \eqref{elem-tau}.}



On the other hand, \cite[Corollaries 5.2, 5.9]{YHSY} and \eqref{010319} ensure
\begin{equation}\label{lim-5}
A^{s_1}_{u_1,q_1}(\Omega) \hookrightarrow  A^{s_1,\tau_1}_{p_1,q_1} (\Omega) \hookrightarrow  A^{s_2,\tau_2}_{p_2,q_2}(\Omega) \hookrightarrow  B^{s_2}_{\infty,\infty}(\Omega),
\end{equation}
such that the compactness of $\id_\tau$ implies $s_1-\frac{d}{u_1}>s_2$ by the well-known classical results. This proves the necessity of the condition. \\

{\em Substep 3.2}. Assume $\tau_1 \ge \frac{1}{p_1}$ with $q_1=\infty$ if $\tau_1 = \frac{1}{p_1}$. We can argue in the same way as above. Let $s_1-s_2>d(\frac{1}{p_1}-\tau_1)$.  We choose $s_0$ such that $s_1-d(\frac{1}{p_1}-\tau_1)>s_0>s_2$. Then by the identities in Proposition~\ref{yy02} and \eqref{ftbt},
\begin{equation}\label{tau2border}
\ate(\Omega)=B^{s_1+d(\tau_1-\frac{1}{p_1})}_{\infty,\infty}(\Omega)\hookrightarrow B^{s_0}_{\infty,\infty}(\Omega)= F^{s_0,\tau_2}_{p_2,\infty}(\Omega) \hookrightarrow \atz(\Omega),
\end{equation}
where the first embedding is compact for $s_1+d(\tau_1-\frac{1}{p_1})>s_0$ and the last embedding is continuous for $s_0>s_2$ by \eqref{elem-0-t} and \eqref{elem-tau}. Hence $\id_\tau$ is compact. Conversely, the  compactness of $\id_\tau$  implies
\begin{equation}\label{lim-9}
B^{s_1+d(\tau_1-\frac{1}{p_1})}_{\infty,\infty}(\Omega)=\ate(\Omega)\hookrightarrow {A}^{s_2,\tau_2}_{p_2,q_2}(\Omega)\hookrightarrow  B^{s_2}_{\infty,\infty}(\Omega)
\end{equation}
where we used \eqref{010319} and Proposition~\ref{yy02}. But the resulting compactness of the outer embedding of Besov spaces leads to the desired condition $ s_1+d(\tau_1-\frac{1}{p_1})>s_2$.\\

\emph{Substep 3.3}. Let finally $\tau_1 = \frac{1}{p_1}$ with $q_1<\infty$. So we are in the double-limiting case and need to show the compactness of $\id_\tau$ if, and only if, $s_1>s_2$. The sufficiency can be obtained via
\begin{equation}\label{lim-7}
\ate(\Omega) \hookrightarrow B^{s_1}_{\infty,\infty}(\Omega) \hookrightarrow B^{s_2}_{\infty,q_2}(\Omega) \hookrightarrow \atz(\Omega)
\end{equation}
where we use \cite[Corollary~5.2]{YHSY} in the last embedding in case of $A=B$, extended by the same argument as above to $A=F$ via \eqref{elem-tau}. The second embedding is compact for $s_1>s_2$.


Conversely, if $\id_\tau$ is compact, choose $q_0\leq \min\{p_1,q_1\}$. Then
\begin{equation}\label{lim-8}
B^{s_1}_{\infty,q_0}(\Omega)\hookrightarrow \ate(\Omega) \hookrightarrow \atz(\Omega) \hookrightarrow B^{s_2}_{\infty,\infty}(\Omega)
\end{equation}
is compact, where we used for the first embedding \cite[Corollary~5.2]{YHSY} (with \eqref{elem-tau} for $A=F$) again, and \eqref{010319} for the last one. But this implies $s_1>s_2$.
 \end{proof}


\begin{remark} Usually one needs the condition $s_1-s_2>0$ to prove compactness of this kind of embeddings. Curiously this is not the case when considering spaces of type $\bt$ and $\ft$. This can easily be seen by condition \eqref{reform}, for instance when $\tau_1\geq \frac{1}{p_1}$ and $\tau_2 < \frac{1}{p_2}$. In parallel to \eqref{010319} and Proposition \ref{yy02}, this evidences the fact that the parameter $\tau$ modifies  indeed  the smoothness of these spaces.
\end{remark}

\begin{remark}\label{comp-hybrid}
  We briefly return to our Remark~\ref{T-hybrid} which referred to the coincidence of Triebel's hybrid spaces $L^r\A$ with the spaces $\at$ if $\tau=\frac1p+\frac{r}{d}$. Obviously, defining both by restriction to $\Omega$, this is transferred to spaces on domains. In that sense Theorem~\ref{comp-tau-u} can be formulated as follows: let $s_i\in \real$, $0<q_i\leq\infty$, $0<p_i< \infty$, $-\frac{d}{p_i}\leq r_i<\infty$, $i=1,2$. Then the embedding
\begin{equation}
	 \id_{L}: L^{r_1}\Ae(\Omega )\hookrightarrow L^{r_2}\Az(\Omega )
\end{equation}
is compact if, and only if,
the following condition holds
\begin{align}
  {s_1-s_2}> & \max\left\{(r_2)_+-(r_1)_+, -r_1+\max\left\{r_2,\frac{p_1}{p_2}\min\{r_1,0\}\right\}\right\},\label{hybrid-1}
  \intertext{i.e.,}
  s_1-s_2 & > \begin{cases} r_2-r_1, & \text{if}\quad r_2\geq 0,\\ -r_1, &\text{if}\quad r_2<0, r_1\geq 0, \\
\max\left\{0, -r_1 + \max\left\{r_2, \frac{p_1}{p_2}r_1\right\}\right\},
    & \text{if}\quad r_2<0, r_1< 0.\nonumber
\end{cases}
\end{align}
\end{remark}


We now collect some immediate consequences of the above compactness result. We begin with the case $\tau_1=\tau_2=\tau\geq 0$.

\begin{corollary}  \label{comp-tau-eq}
  Let  $s_i\in \real$, $0<q_i\leq\infty$, $0<p_i\leq \infty$ (with $p_i<\infty$ in case of $A=F$), $i=1,2$, and $\tau\geq 0$.
    The embedding
\begin{equation} 
	 \id_{\tau}: A^{s_1,\tau}_{p_1,q_1}(\Omega )\hookrightarrow A^{s_2,\tau}_{p_2,q_2}(\Omega )
\end{equation}
is compact if, and only if,
  \begin{equation}\label{same-tau}
    \frac{s_1-s_2}{d} > \begin{cases}
      \frac{1}{p_1}-\frac{1}{p_2}, & \text{if}\quad p_1<p_2, \\[1ex]
     \min\left\{0, \frac{1}{p_1}-\min\left\{\tau,\frac{1}{p_2}\right\}\right\}, &\text{if}\quad  p_1\geq p_2.
    \end{cases}
  \end{equation}
\end{corollary}

\begin{proof}
We apply Theorem~\ref{comp-tau-u} with $\tau_1=\tau_2$.
\end{proof}

\begin{remark}
The result is well-known for $\tau=0$, where \eqref{same-tau} reads as $s_1-s_2>d(\frac{1}{p_1}-\frac{1}{p_2})_+$. We find it interesting that for $\tau>0$ there is not a simple `$\tau$-shift', but an interplay between $\tau$ and the $p_i$-parameters -- however, only when $p_1\geq p_2$. This again refers to the hybrid role played by the additional $\tau$-parameters and makes it even more obvious that it influences {\em both} the smoothness {\em and} the  integrability parameters $s_i$ and $p_i$, respectively.
\end{remark}

Now we deal with special target spaces. In the case of $L_\infty(\Omega)$  and $\bmo(\Omega)$ we have the following result.
\begin{corollary}  \label{MorreyintoLinfty}
  Let  $s\in \real$, $0<p\le u<\infty$ and $q\in(0,\infty]$.  Then the following conditions are equivalent
  \bli
  \item[{\upshape\bfseries (i)\hfill}] the embedding $ \MA(\Omega) \hookrightarrow L_{\infty}(\Omega)$ is compact,
  \item[{\upshape\bfseries (ii)\hfill}] the embedding $\MA(\Omega )\hookrightarrow \bmo(\Omega)$ is compact,
  \item[{\upshape\bfseries (iii)\hfill}]    the following inequality holds $s> \frac{d}{u}$.
  \eli
\end{corollary}

\begin{proof}
  The equivalence of (i) and (iii) was proved in \cite{hs12b,hs14}, whereas the equivalence of (ii) and (iii) follows from Theorem~\ref{comp-tau-u} since $\bmo(\Omega)= F^{0,1/r}_{r,2}(\Omega)$, $0<r<\infty$, and $\MF(\Omega) = F^{s,\tau}_{p,q}(\Omega)$, $0\leq \tau=\frac1p-\frac1u$. This covers the case $\MA=\MF$. The extension to the case $\MA=\MB$ is done via \eqref{elem}.
\end{proof}

The counterpart of Corollary~\ref{MorreyintoLinfty} for spaces of type $\at$ reads as follows.

\begin{corollary}\label{atauinLinf}
 Let $s\in\real$, $\tau\geq 0$, $0< p,q\leq\infty$ (with $p<\infty$ if $A=F$). Then the following conditions are equivalent
  \bli
  \item[{\upshape\bfseries (i)\hfill}] the embedding $ \at(\Omega) \hookrightarrow L_{\infty}(\Omega)$ is compact,
  \item[{\upshape\bfseries (ii)\hfill}] the embedding $\at(\Omega )\hookrightarrow \bmo(\Omega)$ is compact,
  \item[{\upshape\bfseries (iii)\hfill}]  the following inequality holds $s> d\left(\frac1p-\tau\right)$.
 \eli
\end{corollary}

\begin{proof}
{\em Step 1}. We prove the equivalence of (i) and (iii).
  The case $\tau=0$ is well-known, so we assume $\tau>0$.
  Note that the continuity of that embedding was studied in \cite[Proposition~2.18]{HMSS} already, with the outcome that $\at(\Omega) \hookrightarrow L_\infty(\Omega)$ if, and only if, $s>d(\frac1p-\tau)$. Hence (i) implies (iii) and we are left to show the converse.
Assume first, in addition, that $p\geq 2$, and $d(\frac1p-\tau)<s<d(\frac1p-\tau)+1$. We dealt with that situation in \cite[Corollary~5.10]{YHMSY} and characterised the asymptotic behaviour of  approximation numbers of the embedding $\at(\Omega)\hookrightarrow L_\infty(\Omega)$. In particular, that outcome implies (i). The additional restriction for $s$ (from above) can immediately be removed in view of \eqref{elem-0-t} (adapted to spaces on bounded domains). Now let $0<p<2$ and $s>d(\frac1p-\tau)$. We use a Sobolev-type embedding: choose $\sigma=s-d(\frac1p-\frac12)<s$, then $A^{\sigma,\tau}_{2,q}(\Omega)\hookrightarrow L_\infty(\Omega)$ compactly by our previous argument, and $\at(\Omega)\hookrightarrow A^{\sigma,\tau}_{2,q}(\Omega)$ by the Sobolev-type embedding, cf. \cite[Propositions~3.4, 3.8]{YHMSY}  (adapted to spaces on domains).

{\em Step 2}. We prove the equivalence of (ii) and (iii). However, in case of $A=F$ this coincides with Theorem~\ref{comp-tau-u} for $s_1=s$, $s_2=0$, $p_1=p_2=p$, $\tau_1=\tau$, $\tau_2=\frac1p$, $q_1=q$ and $q_2=q$ since $\bmo(\Omega)= F^{0,1/p}_{p,2}(\Omega)$. 
The extension to the case $A=B$ results from \eqref{elem-0-t} and \eqref{elem-tau}.
\end{proof}

\begin{remark}
In case of $\tau>\frac1p$  or  $\tau=\frac1p$ and $q=\infty$, Corollary~\ref{atauinLinf} is well-known as a compact embedding within the scale of Besov spaces $\B(\Omega)$, in view of Proposition~\ref{yy02}. If $0\leq \tau<\frac1p$, then Corollaries~\ref{MorreyintoLinfty} and \ref{atauinLinf} coincide for $\mathcal{E}$-spaces, using \eqref{fte}. In view of our continuity result   \cite[Proposition~2.18]{HMSS} the above outcome can  be reformulated for $\tau>0$ such that $\at(\Omega)\hookrightarrow L_\infty(\Omega)$ is compact if, and only if, it is bounded. This is different from the case $\tau=0$.
\end{remark}

We finally formulate the corresponding results for compact embeddings into $L_r(\Omega)$, $1\leq r<\infty$, which can be seen as the counterparts of Corollaries~\ref{MorreyintoLinfty} and \ref{atauinLinf} where $r=\infty$.

\begin{corollary}\label{atauinLr}
 Let $s\in\real$, $0< p,q\leq\infty$ (with $p<\infty$ if $A=F$), $1\leq r<\infty$.
  \bli
  \item[{\upshape\bfseries (i)\hfill}] Let $u\in [p, \infty)$ (or $p=u=\infty$ if $\mathcal{A}=\mathcal{N}$). Then
\[
\MA(\Omega)\hookrightarrow L_r(\Omega) \quad\text{is compact\quad if, and only if,}\quad
s>\frac{d}{u}\left(1-\frac{p}{r}\right)_+\ .
\]
  \item[{\upshape\bfseries (ii)\hfill}] Let $\tau\geq 0$. Then
\[
\at(\Omega)\hookrightarrow L_r(\Omega)\]
\text{is compact if, and only if,}
\begin{equation}\label{lim-12}
  s>d \begin{cases} \left(\frac1p-\tau\right) & \text{if}\quad \tau\geq \frac1p, \\[1ex] \left(\frac1p-\tau\right)\left(1-\frac{p}{r}\right)_+ & \text{if}\quad \tau\leq \frac1p.\end{cases}
\end{equation}
 \eli
\end{corollary}

\begin{proof}
  Case (i) was already shown in \cite[Proposition~5.3]{hs12b} (for $\mathcal{A}=\mathcal{N}$). In view of \eqref{elem} and the independence of the condition with respect to $q$ the counterpart for $\mathcal{A}=\mathcal{E}$ follows (and slightly extends our recent result in \cite[Corollary~5.4]{hs14} to $r=1$). We come to (ii) and start with the case $A=B$. Note that $B^0_{r,1}(\Omega)\hookrightarrow L_r(\Omega)\hookrightarrow B^0_{r,\infty }(\Omega)$, $1\leq r\leq\infty$, so we apply  Theorem~\ref{comp-tau-u} for $s_1=s$, $s_2=0$, $p_1=p$, $p_2=r$, $\tau_1=\tau$, $\tau_2=0$, $q_1=q$, and $q_2=1$ or $q_2=\infty$ to obtain the necessary and sufficient conditions. Again we benefit from the independence of \eqref{tau-comp-u2} with respect to the $q$-parameters. Finally, the case $A=F$ follows by \eqref{elem-tau} again.
\end{proof}

\begin{corollary}  \label{cont-tau}
  Let  $s_i\in \real$, $0<q_i\leq\infty$, $0<p_i\leq \infty$ (with $p_i<\infty$ in case of $A=F$), $\tau_i\geq 0$, $i=1,2$.
    There is no continuous embedding
\[
	 \id_{\tau}: \ate(\Omega )\hookrightarrow \atz(\Omega )
\]
if
\begin{equation}\label{no-cont-tau}
  \frac{s_1-s_2}{d} < \critical.
\end{equation}
\end{corollary}

\begin{proof}
Here we directly follow our proof of Theorem~\ref{comp-tau-u} and apply our continuity results \cite[Theorem~3.1]{hs12b} (for $\mathcal{N}$-spaces) and \cite[Theorem~5.2]{hs14} (for $\mathcal{E}$-spaces).
  We again follow the splitting suggested by \eqref{reform}. So let us assume in the sequel that there is a continuous embedding $\id_\tau$.

{\em Step 1}. Let $\tau_2\geq \frac{1}{p_2}$ with $q_2=\infty$ if $\tau_2=\frac{1}{p_2}$. If also $\tau_1\geq \frac{1}{p_1}$ with $q_1=\infty$ if $\tau_1=\frac{1}{p_2}$, then \eqref{tlarge2} implies the continuity of
\[\id: B^{s_1+d(\tau_1-\frac{1}{p_1})}_{\infty,\infty}(\Omega) \hookrightarrow B^{s_2+d(\tau_2-\frac{1}{p_2})}_{\infty,\infty}(\Omega)\]
which is well-known to imply
  \begin{equation}\label{reform-1a}
    s_1+d\left(\tau_1-\frac{1}{p_1}\right)\geq s_2+d\left(\tau_2-\frac{1}{p_2}\right),
  \end{equation}
contradicting \eqref{no-cont-tau} in that case.
Likewise, if $ \tau_1\leq\frac{1}{p_1}$ with $q_1<\infty$ if $\tau_1=\frac{1}{p_1}$, then \eqref{below-tau-small} leads to the continuity of
\[\id: B^{s_1+d(\tau_1-\frac{1}{p_1}+\frac{1}{p_0})}_{p_0,q_0}(\Omega)\hookrightarrow  B^{s_2+d(\tau_2-\frac{1}{p_2})}_{\infty,\infty}(\Omega)\]
 which implies again \eqref{reform-1a} and thus contradicts \eqref{no-cont-tau}.

{\em Step 2}. Assume $\tau_2< \frac{1}{p_2}$ and proceed parallel to Step~2 of the proof of Theorem~\ref{comp-tau-u}. If also $\tau_1<\frac{1}{p_1}$, then the continuity of $\id_\tau:\fte(\Omega)\hookrightarrow\ftz(\Omega)$ results in the continuity of the corresponding embedding between $\mathcal{E}$-spaces, using \eqref{fte}, which in turn by \cite[Theorem~5.2]{hs14} leads to a contradiction of \eqref{reform-1a} again. The extension to spaces $\bt(\Omega)$ is obtained via \eqref{N-BT-emb}, \eqref{N-BT-equal},
\[
\MBe(\Omega)\hookrightarrow \bte(\Omega)\hookrightarrow\btz(\Omega)\hookrightarrow B^{s_2,\tau_2}_{p_2,\infty}(\Omega)=\mathcal{N}^{s_2}_{u_2,p_2,\infty}(\Omega),
\]
which by \cite[Theorem~3.1]{hs12b} implies
\[\frac{s_1-s_2}{d}\geq      \max\left\{0, \frac{1}{p_1}-\tau_1-\frac{1}{p_2}+\max\left\{\tau_2,\frac{p_1}{p_2}\tau_1\right\}\right\} = \critical
\]
contradicting \eqref{no-cont-tau} in this setting. If $\tau_1\geq \frac{1}{p_1}$ with $q_1=\infty$ when $\tau_1=\frac{1}{p_1}$, then Proposition~\ref{yy02} together with \eqref{N-BT-equal}, \eqref{fte}, \eqref{elem} yield the continuity of
\[
\id: B^{s_1+d(\tau_1-\frac{1}{p_1})}_{\infty,\infty}(\Omega) \hookrightarrow \mathcal{N}^{s_2}_{u_2,p_2,\infty}(\Omega)
\]
and thus again $s_1-s_2\geq d(\frac{1}{p_1}-\tau_1)$, contradicting \eqref{no-cont-tau}. When $\tau_1=\frac{1}{p_1}$, $q_1<\infty$, \eqref{lim-2} leads to the same contradiction again.

{\em Step 3}. Assume finally $\tau_2=\frac{1}{p_2}$, $q_2<\infty$. Then the embeddings \eqref{lim-5}, \eqref{lim-9} and \eqref{lim-8} disprove \eqref{no-cont-tau} in the corresponding settings.
\end{proof}

\begin{remark}
Obviously Theorem~\ref{comp-tau-u} implies, in particular, that the embedding $\id_\tau$ is continuous when $s_1-s_2> d\ \critical$. In view of Corollary~\ref{cont-tau} it thus turns out that the limiting case for the embedding $\id_\tau$ is indeed
\[
  \frac{s_1-s_2}{d}  = \critical.
\]
Here some influence of the fine parameters $q_i$ can also be expected. But this question is postponed to a separate study in the future.
\end{remark}

\section{Entropy numbers}

First we return to the compact embedding $\id_{\mathcal{A}}$ given by \eqref{bd1comp}, recall Theorem~\ref{comp}. For its entropy numbers we obtained in \cite[Corollaries~4.1, 4.3]{HaSk-morrey-comp} the following result.


\begin{theorem}  \label{ek-NE}
  Let  $s_i\in \real$, $0<q_i\leq\infty$, $0<p_i\leq u_i<\infty$, {or $p_i=u_i=\infty$ in the case of $\mathcal N$-spaces},  $i=1,2$. Assume that \eqref{bd3acomp} is satisfied. Then we obtain for the entropy numbers of the compact embedding
\begin{equation*}
	 \id_{\mathcal{A}}: \MAe(\Omega )\hookrightarrow \MAz(\Omega )
\end{equation*}
the following results:
\bli
\item[{\hfill\bfseries\upshape (i)\hfill}]
If
\begin{equation}
\frac{1}{p_1} -\frac{1 }{p_2}\ \geq\ \frac{s_1-s_2}{d}\ >\  \frac{p_1}{u_1} \Big(\frac{1}{p_1}- \frac{1}{p_2}\Big) \qquad\text{and}\qquad \frac{u_2}{p_2}<\frac{u_1}{p_1},
\end{equation}
then there exists some $c>0$ and for any $\varepsilon>0$ some $c_\varepsilon>0$ such that for all $k\in\nn$,
\begin{equation}\label{new_morrey_fu}
  c \ k^{- \frac{u_1}{u_1-p_1}(\frac{s_1-s_2}{d}-\frac{p_1}{u_1}(\frac{1}{p_1}- \frac{1}{p_2}) )} \le \   e_k \ \big(\id_{\mathcal{A}}\big)\ \le \ c_\varepsilon k^{- \frac{u_1}{u_1-p_1}(\frac{s_1-s_2}{d}-\frac{p_1}{u_1}(\frac{1}{p_1}- \frac{1}{p_2}) )+\varepsilon}.
 \end{equation}
\item[{\hfill\bfseries\upshape (ii)\hfill}]
In all other cases admitted by \eqref{bd3acomp}, it holds
\begin{equation}\label{ek-class-fu}
    e_k \big(\id_{\mathcal{A}}: \MAe(\Omega)\hookrightarrow \MAz(\Omega)\big) \sim  k^{-\frac{s_1-s_2}{d}},\quad k\in\nn.
\end{equation}
\eli
\end{theorem}
{
\begin{proof} The cases when $0<p_i\leq u_i <\infty$, $i=1,2$, were proved in \cite{HaSk-morrey-comp}, as already mentioned. So it remains to verify the cases $p_1=u_1=\infty$ or $p_2=u_2=\infty$. Note that in both cases we are in part (ii).

If $p_1=u_1=\infty$, \eqref{bd3acomp} reads as $s_1-s_2>0$. Now we use \eqref{pu-1-infty-1} and the multiplicativity of the entropy numbers to obtain
\begin{equation}
e_k(\id_{\mathcal{A}}) \ \lesssim \ e_k\left( B^{s_1}_{\infty, q_1}(\Omega) \hookrightarrow A^{s_2}_{u_2, q_2}(\Omega)\right)\ \lesssim \ k^{-\frac{s_1-s_2}{d}}.\nonumber
\end{equation}
Moreover, by \eqref{pu-1-infty-2} we get the desired estimate from below:
\begin{equation}
e_k(\id_{\mathcal{A}})\ \gtrsim \ e_k\left( B^{s_1}_{\infty, q_1}(\Omega) \hookrightarrow A^{s_2}_{p_2, q_2}(\Omega)\right)\ \gtrsim \ k^{-\frac{s_1-s_2}{d}}. \nonumber
\end{equation}
The case when $p_2=u_2=\infty$ and $\frac{s_1-s_2}{d}>\frac{1}{u_1}$ follows similarly but using \eqref{pu-2-infty-1}-\eqref{pu-2-infty-3} this time. 
\end{proof}
}

Now we give the counterpart of the above result for the compact embedding \eqref{tau-comp-u1}, described by Theorem~\ref{comp-tau-u}.

\begin{theorem}  \label{ek-tau}
  Let  $s_i\in \real$, $0<q_i\leq\infty$, $0<p_i\leq \infty$ (with $p_i<\infty$ in case of $A=F$), $\tau_i\geq 0$, $i=1,2$.
 Assume that \eqref{tau-comp-u2} is satisfied. Then we obtain for the entropy numbers of the compact embedding
\begin{equation*}
	 \id_{\tau}: \ate(\Omega )\hookrightarrow \atz(\Omega )
\end{equation*}
the following results:
\bli
\item[{\hfill\bfseries\upshape (i)\hfill}]
If $\tau_1<\frac{1}{p_1}$,
\begin{equation}\label{tau-comp-7a}
\tau_1\, \frac{p_1}{p_2} > \tau_2, 
\end{equation}
and
\begin{equation}\label{tau-comp-7b}
\frac{1}{p_1} -\frac{1 }{p_2}\ \geq\ \frac{s_1-s_2}{d}\ >\ (1- p_1\tau_1) \Big(\frac{1}{p_1}- \frac{1}{p_2}\Big),
\end{equation}
then there exists some $c>0$ and for any $\varepsilon>0$ some $c_\varepsilon>0$ such that for all $k\in\nn$,
\begin{equation}\label{tau-comp-9}
  c\ k^{- \frac{1}{p_1 \tau_1}(\frac{s_1-s_2}{d}-(1-p_1\tau_1)(\frac{1}{p_1}- \frac{1}{p_2}) )} \ \le \ e_k\left(\id_\tau\right) \ \le \ c_\varepsilon \ k^{- \frac{1}{p_1\tau_1}(\frac{s_1-s_2}{d}-(1-p_1\tau_1)(\frac{1}{p_1}- \frac{1}{p_2}) )+\varepsilon}.
\end{equation}
\item[{\hfill\bfseries\upshape (ii)\hfill}]
In all other cases admitted by \eqref{tau-comp-u2}, it holds
\begin{equation}\label{tau-comp-3}
e_k(\id_\tau) \ \sim \ k^{-\frac{s_1-s_2}{d} - (\tau_1-\frac{1}{p_1})_+ + (\tau_2-\frac{1}{p_2})_+},\quad k\in\nn.
\end{equation}
\eli
\end{theorem}


\begin{proof} {To prove this result we basically rely on the proof of Theorem \ref{tau-comp-u1} and the corresponding counterparts of the entropy numbers for the spaces $\MA$, cf. Theorem \ref{ek-NE}, as well as for the classical spaces $\A$, cf. Remark \ref{remark-ek-ak-class}. Therefore, we start by proving part (i) and then we split part (ii) in all the possible cases as in the proof of Theorem \ref{tau-comp-u1}.}

{\em Step 1}. We deal with case (i). Assumption \eqref{tau-comp-7a} implies also $\tau_2<\frac{1}{p_2}$,  such that in view of our reformulation \eqref{reform} of the compactness condition \eqref{tau-comp-u2} it is obvious that the expression on the right-hand side of \eqref{tau-comp-u2} equals $(1-p_1\tau_1)(\frac{1}{p_1}-\frac{1}{p_2})$. We first use \eqref{fte} to get
\begin{equation*}
e_k(\id_\tau:\fte(\Omega)\hookrightarrow \ftz(\Omega))\ \sim \ e_k\left(\id_{\mathcal{E}}: \MFe(\Omega)\hookrightarrow \MFz(\Omega)\right),\quad \frac{1}{u_i} = \frac{1}{p_i}-\tau_i, \quad i=1,2,
\end{equation*}
together with Theorem~\ref{ek-NE}(i) which covers the case $A=F$. As for the $B$-case, we benefit from \eqref{tau2small-1} (to the estimate from above) and \eqref{tau2small-2} (to the estimate from below),
such that the multiplicativity of entropy numbers implies
\[
e_k\left(\id_{\mathcal{N}}: \MBe(\Omega)\hookrightarrow \mathcal{N}^{s_2}_{u_2,p_2,\infty}(\Omega)\right)\ \lesssim \ e_k\left(\id_\tau\right)\ \lesssim \ e_k\left(\id_{\mathcal{N}}:\mathcal{N}^{s_1}_{u_1,p_1,\infty}(\Omega) \hookrightarrow \MBz(\Omega)\right).
\]
Application of Theorem~\ref{ek-NE}(i) concludes the argument.\\

{\em Step 2}. We are left to prove \eqref{tau-comp-3} in all remaining cases.

{\em Substep 2.1}. We first continue with the case $\tau_i<\frac{1}{p_i}$, $i=1,2$, from Step~1, where we now assume that \eqref{tau-comp-7a} or \eqref{tau-comp-7b} are not satisfied. In both cases we proceed as above and use the coincidence \eqref{fte} together with Theorem~\ref{ek-NE}(ii). This yields
\[
e_k\left(\id_\tau:\fte(\Omega)\hookrightarrow \ftz(\Omega)\right)\ \sim \ k^{-\frac{s_1-s_2}{d}} \ = \ k^{-\frac{s_1-s_2}{d} - (\tau_1-\frac{1}{p_1})_+ + (\tau_2-\frac{1}{p_2})_+}, \quad k\in\nn,
\]
i.e. the desired result \eqref{tau-comp-3} for $A=F$. The case $A=B$ is done as in the end of Step~1, using again \eqref{tau2small-1} and \eqref{tau2small-2}. 

{\em Substep 2.2}. We stick to $\tau_2<\frac{1}{p_2}$, now together with $\tau_1\geq \frac{1}{p_1}$ and $q_1=\infty$ if $\tau_1=\frac{1}{p_1}$. In view of \eqref{reform} we thus assume $s_1-s_2 > d(\frac{1}{p_1}-\tau_1)$ and need to show \eqref{tau-comp-3} in the form
\[
e_k(\id_\tau)\ \sim \ k^{-\frac{s_1-s_2}{d}-(\tau_1-\frac{1}{p_1})}, \quad k\in\nn.
\]
This can be seen as follows. By \eqref{tau2small-4},
\[
e_k(\id_\tau: \ate(\Omega)\hookrightarrow \atz(\Omega)) \ \gtrsim \ e_k\left(\id_{\mathcal{N}}: \mathcal{N}^{s_1+d(\tau_1-\frac{1}{p_1})}_{\infty,\infty,\infty}(\Omega)\hookrightarrow \mathcal{N}^{s_2}_{u_2, p_2, \infty}(\Omega)\right),
\]
which in view of Theorem~\ref{ek-NE} (applied to $\mathcal{A}=\mathcal{N}$) leads to
\[
e_k(\id_\tau: \ate(\Omega)\hookrightarrow \atz(\Omega)) \ \gtrsim \  k^{-\frac{s_1-s_2}{d}-(\tau_1-\frac{1}{p_1})}, \quad k\in\nn,
\]
under the given assumptions. The estimate from above follows similarly from \eqref{tau2small-3} and again the multiplicativity of entropy numbers.

{\em Substep 2.3}. We study the case $\tau_2\geq \frac{1}{p_2}$ with $q_2=\infty$ if $\tau_2=\frac{1}{p_2}$. According to \eqref{reform} we assume that
  \begin{equation}\label{tau-comp-2}
  \frac{s_1-s_2}{d} > \frac{1}{p_1}-\tau_1-\frac{1}{p_2}+\tau_2.
\end{equation}
If $\tau_1\geq \frac{1}{p_1}$ with $q_1=\infty$ if $\tau_1=\frac{1}{p_1}$, then Proposition~\ref{yy02} implies that
\[
e_k(\id_\tau)\ \sim \ e_k\big( \id_B : B^{\sigma_1}_{\infty,\infty}(\Omega) \hookrightarrow B^{\sigma_2}_{\infty,\infty}(\Omega)\big),\quad \sigma_i=s_i+d\left(\tau_i-\frac{1}{p_i}\right),\quad i=1,2.
\]
But the asymptotic behaviour of the entropy numbers in that latter case is well-known for $\sigma_1>\sigma_2$, which is equivalent to \eqref{tau-comp-2}, that is
\[
e_k\big( \id_B : B^{\sigma_1}_{\infty,\infty}(\Omega) \hookrightarrow B^{\sigma_2}_{\infty,\infty}(\Omega)\big) \ \sim \  k^{-\frac{\sigma_1-\sigma_2}{d}}, \qquad k\in\nn,
 \]
cf. Remark \ref{remark-ek-ak-class}. This coincides with \eqref{tau-comp-3} in that case.

If $\tau_1<\frac{1}{p_1}$ we argue as follows. We use \eqref{fte}, \eqref{MB=B}
and Proposition~\ref{yy02} to get
\begin{equation}\label{tau-comp-8}
e_k(\id_\tau: \fte(\Omega)\hookrightarrow\ftz(\Omega)) \ \sim \ e_k\left(\id: \MFe(\Omega) \hookrightarrow \mathcal{N}^{s_2+d(\tau_2-\frac{1}{p_2})}_{\infty,\infty,\infty}(\Omega)\right).
\end{equation}
On the other hand, Theorem~\ref{ek-NE}(ii) yields
\[
e_k\left(\id_\mathcal{N}: \MBe(\Omega) \hookrightarrow \mathcal{N}^{s_2+d(\tau_2-\frac{1}{p_2})}_{\infty,\infty,\infty}(\Omega)\right)\ \sim \ k^{-\frac{s_1-s_2}{d} + \tau_2-\frac{1}{p_2}},\quad k\in\nn,
\]
whenever \eqref{tau-comp-2} is satisfied. In view of \eqref{elem} we can thus continue \eqref{tau-comp-8} by
\[
e_k(\id_\tau: \fte(\Omega)\hookrightarrow\ftz(\Omega)) \ \sim \
k^{-\frac{s_1-s_2}{d} + \tau_2-\frac{1}{p_2}},\quad k\in\nn,
\]
which is the desired result in case of $A=F$. The case $A=B$ can be obtained noting that
\[
e_k(\id_\tau: B^{s_1, \tau_1}_{p_1, \infty}(\Omega)\hookrightarrow\btz(\Omega)) \ \sim \ e_k\left(\id_{\mathcal{N}}: \mathcal{N}^{s_1}_{u_1, p_1, \infty}(\Omega) \hookrightarrow \mathcal{N}^{s_2+d(\tau_2-\frac{1}{p_2})}_{\infty,\infty,\infty}(\Omega)\right),
\]
where we have used \eqref{MB=B} and \eqref{N-BT-equal}. Now \eqref{elem-1-t} and Theorem \ref{ek-NE}(ii) give the desired result. \\

{\em Step 3}. It remains to study the limiting cases, that is, when $\tau_i=\frac{1}{p_i}$ and $q_i<\infty$ for $i=1$ or $i=2$.

{\em Substep 3.1}. Let $\tau_1=\frac{1}{p_1}$, $q_1<\infty$, and $\tau_2<\frac{1}{p_2}$. Following the arguments of Substep 2.3 of the proof of Theorem~\ref{comp-tau-u}, in particular, \eqref{lim-1} and \eqref{lim-2}, and using the multiplicativity of entropy numbers we arrive at
\begin{equation} \label{lim-3}
e_k\left(\id: B^{s_1}_{\infty,q_1}(\Omega) \hookrightarrow \mathcal{N}^{s_2}_{u_2,p_2,\infty}(\Omega)\right)\ \lesssim \
e_k(\id_\tau)\ \lesssim \ e_k\left(\id: B^{s_1}_{\infty,\infty}(\Omega) \hookrightarrow \MAz(\Omega)\right),
\end{equation}
where $ \frac{1}{u_2}=\frac{1}{p_2}-\tau_2>0$.
However, completely parallel to Substeps~2.2 and 2.3, Theorem~\ref{ek-NE}(ii) implies
\[
e_k\left(\id: B^{s_1}_{\infty,\infty}(\Omega) \hookrightarrow \MAz(\Omega)\right) \ \sim \
e_k\left(\id: B^{s_1}_{\infty,q_1}(\Omega) \hookrightarrow \mathcal{N}^{s_2}_{u_2,p_2,\infty}(\Omega)\right) \ \sim \ k^{-\frac{s_1-s_2}{d}}
\]
whenever $s_1>s_2$, such that \eqref{lim-3} 
finally results in
\[
e_k(\id_\tau) \ \sim \ k^{-\frac{s_1-s_2}{d}},\quad k\in\nn.
\]

{\em Substep 3.2}. Let $\tau_1=\frac{1}{p_1}$, $q_1<\infty$, and $\tau_2\geq \frac{1}{p_2}$ with $q_2=\infty$ if $\tau_2=\frac{1}{p_2}$. In view of \eqref{reform}, we assume now $s_1-s_2 > d(\tau_2-\frac{1}{p_2})$. Then \eqref{tlarge1} leads to
\[
e_k(\id_\tau) \lesssim \ e_k\left(\id_B: B^{s_1}_{\infty,\infty}(\Omega) \hookrightarrow B^{s_2+d(\tau_2-\frac{1}{p_2})}_{\infty,\infty}(\Omega)\right) \ \lesssim \ k^{-\frac{s_1-s_2}{d}+(\tau_2-\frac{1}{p_2})}, \quad k\in\nn,
\]
by the same arguments as above. Conversely, according to \cite[Corollary~5.2]{YHSY} (adapted to spaces on bounded domains), Proposition \ref{yy02} and \eqref{elem-tau}, we have in this case
\begin{equation}\label{lim-13}
B^{s_1}_{\infty,q_1}(\Omega) \hookrightarrow \ate(\Omega) \hookrightarrow \atz(\Omega) = B^{s_2+d(\tau_2-\frac{1}{p_2})}_{\infty,\infty}(\Omega),
\end{equation}
such that
\[
e_k(\id_\tau) \ \gtrsim \ e_k\left(\id_B:
B^{s_1}_{\infty,q_1}(\Omega) \hookrightarrow B^{s_2+d(\tau_2-\frac{1}{p_2})}_{\infty,\infty}(\Omega)\right)\ \gtrsim \ k^{-\frac{s_1-s_2}{d}+(\tau_2-\frac{1}{p_2})}, \quad k\in\nn.
\]
This concludes the proof in this case.

{\em Substep 3.3}. Let $\tau_2=\frac{1}{p_2}$, $q_2<\infty$, and assume $\tau_1<\frac{1}{p_1}$. The chain of embeddings \eqref{lim-5} leads to
\[
e_k(\id_\tau)\ \gtrsim \ e_k\left(\id: A^{s_1}_{u_1,q_1}(\Omega)\hookrightarrow B^{s_2}_{\infty,\infty}(\Omega)\right)\gtrsim \ k^{-\frac{s_1-s_2}{d}},\quad k\in\nn,
\]
as desired, where we made use of the condition $\frac{s_1-s_2}{d}>\frac{1}{u_1}=\frac{1}{p_1}-\tau_1$ as \eqref{reform} reads in this case.

For the estimate from above we use \eqref{lim-6}. In the same way as there, if $A=B$ we take $q_0=q_2$, and in case $A=F$ we choose $q_0\leq \min\{p_2,q_2\}$ and further use \eqref{elem-tau}. 
Thus
\[
e_k(\id_\tau:\ate(\Omega)\hookrightarrow \atz(\Omega)) \ \lesssim \ e_k\left(\id: \mathcal{N}^{s_1}_{u_1,p_1,\infty}(\Omega) \hookrightarrow \mathcal{N}^{s_2}_{\infty,\infty,q_0}(\Omega)\right) \ \lesssim \ k^{-\frac{s_1-s_2}{d}},\quad k\in\nn,
\]
by Theorem~\ref{ek-NE}(ii).

{\em Substep 3.4}. Let $\tau_2=\frac{1}{p_2}$, $q_2<\infty$, and $\tau_1\geq \frac{1}{p_1}$ with $q_1=\infty$ if $\tau_1=\frac{1}{p_1}$. Assume that \mbox{$s_1-s_2>d(\tau_1-\frac{1}{p_1})$}. We benefit from \eqref{lim-9} to conclude that
\[
e_k(\id_\tau) \ \gtrsim \ e_k\left(\id_B: B^{s_1+d(\tau_1-\frac{1}{p_1})}_{\infty,\infty}(\Omega)\hookrightarrow B^{s_2}_{\infty,\infty}(\Omega)\right) \gtrsim \ k^{-\frac{s_1-s_2}{d} -(\tau_1-\frac{1}{p_1})}, \quad k\in\nn.
\]
For the estimate from above we adapt \eqref{lim-6} properly by
\begin{equation*}
\ate(\Omega)=B^{s_1+d(\tau_1-\frac{1}{p_1})}_{\infty,\infty}(\Omega) = \mathcal{N}^{s_1+d(\tau_1-\frac{1}{p_1})}_{\infty,\infty,\infty}(\Omega) \hookrightarrow \mathcal{N}^{s_2}_{\infty,\infty,q_0}(\Omega)=B^{s_2}_{\infty,q_0}(\Omega)\hookrightarrow  A^{s_2, \tau_2}_{p_2, q_2}(\Omega),
\end{equation*}
with the same choice of $q_0$ as before. Then we get
\[
e_k(\id_\tau)\ \lesssim \ e_k\left(\id_{\mathcal{N}}:
\mathcal{N}^{s_1+d(\tau_1-\frac{1}{p_1})}_{\infty,\infty,\infty}(\Omega) \hookrightarrow \mathcal{N}^{s_2}_{\infty,\infty,q_2}(\Omega)\right)\ \lesssim \
k^{-\frac{s_1-s_2}{d} -(\tau_1-\frac{1}{p_1})}, \quad k\in\nn.
\]

{\em Substep 3.5}. In the double-limiting case, that is, when $\tau_i=\frac{1}{p_i}$, $q_i<\infty$, $i=1,2$, then \eqref{lim-7} and \eqref{lim-8} immediately imply
\[
e_k(\id_\tau)\ \sim \ k^{-\frac{s_1-s_2}{d}}, \quad k\in\nn,
\]
since
\[
e_k\left(\id_B: B^{s_1}_{\infty,\infty}(\Omega) \hookrightarrow B^{s_2}_{\infty,q_2}(\Omega)\right) \sim e_k\left(\id_B: B^{s_1}_{\infty,q_0}(\Omega)\hookrightarrow B^{s_2}_{\infty,\infty}(\Omega)\right) \sim k^{-\frac{s_1-s_2}{d}}
\]
whenever $s_1>s_2$, cf. Remark \ref{remark-ek-ak-class}.
\end{proof}


\begin{remark}\label{ek-hybrid}
  We return to our Remarks~\ref{T-hybrid} and \ref{comp-hybrid} and formulate the result for Triebel's hybrid spaces $L^r\A(\Omega)$. Let $s_i\in \real$, $0<q_i\leq\infty$, $0<p_i< \infty$, $-\frac{d}{p_i}\leq r_i<\infty$, $i=1,2$, satisfy \eqref{hybrid-1}. Then we obtain for the compact embedding
\begin{equation}
	 \id_{L}: L^{r_1}\Ae(\Omega )\hookrightarrow L^{r_2}\Az(\Omega )
\end{equation}
the following results:
\bli
\item[{\hfill\bfseries\upshape (i)\hfill}]
  If $r_1<0$, $r_1p_1 > r_2p_2$, and
\[
\frac{1}{p_1} -\frac{1 }{p_2}\ \geq\ \frac{s_1-s_2}{d}\ >\ - \frac{p_1r_1}{d} \Big(\frac{1}{p_1}- \frac{1}{p_2}\Big),
\]
then there exists some $c>0$ and for any $\varepsilon>0$ some $c_\varepsilon>0$ such that for all $k\in\nn$,
\[
  c \ k^{- \frac{1}{d+p_1r_1}({s_1-s_2} + p_1r_1(\frac{1}{p_1}- \frac{1}{p_2}) )} \ \le  \  e_k \big(\id_L\big)\ \le \ c_\varepsilon \ k^{- \frac{1}{d+p_1r_1}({s_1-s_2} + p_1r_1(\frac{1}{p_1}- \frac{1}{p_2}) )+\varepsilon}.
\]
\item[{\hfill\bfseries\upshape (ii)\hfill}]
In all other cases admitted by \eqref{hybrid-1}, it holds
\[
e_k(\id_L) \ \sim \ k^{-\frac{s_1-s_2}{d} - \frac{(r_1)_+}{d} + \frac{(r_2)_+}{d}},\quad k\in\nn.
\]
\eli
\end{remark}

Parallel to the end of Section~\ref{sect-comp} we now collect some consequences and special cases of Theorem~\ref{ek-tau}. We begin with the counterpart of Corollary~\ref{comp-tau-eq}, that is, when $\tau_1=\tau_2\geq 0$.

\begin{corollary}  \label{ek-tau-eq}
  Let  $s_i\in \real$, $0<q_i\leq\infty$, $0<p_i\leq \infty$ (with $p_i<\infty$ in case of $A=F$), $i=1,2$, and $\tau\geq 0$. Assume that \eqref{same-tau} is satisfied. Then we obtain for the entropy numbers of the compact embedding
\begin{equation} 
	 \id_{\tau}: A^{s_1,\tau}_{p_1,q_1}(\Omega )\hookrightarrow A^{s_2,\tau}_{p_2,q_2}(\Omega )
\end{equation}
that
\[
e_k(\id_\tau: A^{s_1,\tau}_{p_1,q_1}(\Omega )\hookrightarrow A^{s_2,\tau}_{p_2,q_2}(\Omega )) \ \sim \ k^{-\frac{s_1-s_2}{d} - (\tau-\frac{1}{p_1})_+ + (\tau-\frac{1}{p_2})_+},\quad k\in\nn.
\]
\end{corollary}

\begin{proof}
This follows immediately from an application of Theorem~\ref{ek-tau} with $\tau_1=\tau_2=\tau$. Note that part (i) of Theorem~\ref{ek-tau} cannot appear in this setting.
\end{proof}

\begin{remark}
The above result is again well-known for $\tau=0$. We find it interesting to note that for sufficiently small $\tau$, that is, when $0\leq \tau\leq \min\{\frac{1}{p_1}, \frac{1}{p_2}\}$, the asymptotic behaviour for the entropy numbers remains the same,
\[
e_k\left(\id_A: \Ae(\Omega) \hookrightarrow \Az(\Omega)\right) \sim e_k\left(\id_\tau: A^{s_1,\tau}_{p_1,q_1}(\Omega) \hookrightarrow A^{s_2,\tau}_{p_2,q_2}(\Omega)\right) \sim k^{-\frac{s_1-s_2}{d}},\ \ 0\leq \tau\leq \min\left\{\frac{1}{p_1}, \frac{1}{p_2}\right\}.
\]
For sufficiently large $\tau$, also the $\tau$-dependence disappears, due to the coincidence stated in Proposition~\ref{yy02}, that is,
\[
e_k\left(\id_\tau: A^{s_1,\tau}_{p_1,q_1}(\Omega) \hookrightarrow A^{s_2,\tau}_{p_2,q_2}(\Omega)\right) \ \sim \ k^{-\frac{s_1-s_2}{d}+\frac{1}{p_1}-\frac{1}{p_2}},\quad  \tau\geq
 \max\left\{\frac{1}{p_1}, \frac{1}{p_2}\right\}.
\]
So only `in between', that is, for $\min\{\frac{1}{p_1}, \frac{1}{p_2}\} < \tau <  \max\{\frac{1}{p_1}, \frac{1}{p_2}\}$ 
the Morrey parameter $\tau$ 
influences the asymptotic behaviour of entropy numbers. Plainly, for $p_1=p_2$ this case is impossible, so in that case we would really have some `$\tau$-shift',
\[
e_k\left(\id_A: A^{s_1}_{p,q_1}(\Omega) \hookrightarrow A^{s_2}_{p,q_2}(\Omega)\right) \ \sim \
e_k\left(\id_\tau: A^{s_1,\tau}_{p,q_1}(\Omega) \hookrightarrow A^{s_2,\tau}_{p,q_2}(\Omega)\right) \ \sim \ k^{-\frac{s_1-s_2}{d}},\quad  \tau\geq 0.
\]
\end{remark}

We come to the counterpart of Corollary~\ref{MorreyintoLinfty}.

\begin{corollary}  \label{ek-Morrey-Linfty}
  Let  $s\in \real$, $0<p\le u<\infty$, $q\in(0,\infty]$, and $s>\frac{d}{u}$.  Then the entropy numbers of the compact embeddings $\id : \MA(\Omega) \hookrightarrow L_{\infty}(\Omega)$ and $\id : \MA(\Omega)\hookrightarrow \bmo(\Omega)$ behave like
\[
e_k \left(\id:\MA(\Omega) \hookrightarrow L_{\infty}(\Omega)\right)\ \sim \ e_k\left(\id: \MA(\Omega)\hookrightarrow \bmo(\Omega)\right) \ \sim \ k^{-\frac{s}{d}},\quad k\in\nn.
\]
\end{corollary}

\begin{proof}
We can apply part (ii) of Theorem~\ref{ek-tau} due to the coincidence $\bmo(\Omega)= F^{0,1/r}_{r,2}(\Omega)$, $0<r<\infty$, and the well-known embeddings $B^0_{\infty,1}(\Omega)\hookrightarrow L_\infty(\Omega)\hookrightarrow B^0_{\infty,\infty}(\Omega)$, i.e. we work with the assumption $\tau_1p_1< 1=\tau_2p_2$ or $\tau_2=\frac{1}{p_2}=0$.


{
{\em Step 1}. Let us first consider the embedding $\MA(\Omega)\hookrightarrow \bmo(\Omega)$. We get the result for $\mathcal{A}=\mathcal{E}$ by applying Theorem~\ref{ek-tau}(ii) with $\MF(\Omega) = F^{s,\tau}_{p,q}(\Omega)$, $0\leq \tau=\frac1p-\frac1u$, that is
\begin{equation*}
e_k \left(\id : \MF(\Omega) \hookrightarrow \bmo(\Omega)\right) \ \sim \ e_k \left(\id_\tau: \ft (\Omega) \hookrightarrow F^{0, \frac{1}{r}}_{r,2}(\Omega)\right) \ \sim \ k^{-\frac{s}{d}}, \quad k \in \nn.
\end{equation*}
As for the case $\mathcal{A}=\mathcal{N}$, we use \eqref{N-BT-equal}, \eqref{ftbt} and \eqref{MB=B} and obtain for $k \in \nn$ 
\[
e_k \left( \id_\tau : B^{s,0}_{u,q}(\Omega) \hookrightarrow B^{0,\frac12}_{2,2}(\Omega)\right) \ \lesssim \ e_k\left(\id : \MB(\Omega) \hookrightarrow \bmo(\Omega)\right) \ \lesssim \ e_k \left(\id_\tau: B^{s, \tau}_{p, \infty} \hookrightarrow B^{0, \frac{1}{2}}_{2,2}(\Omega)\right), 
\]
which in view of Theorem~\ref{ek-tau}(ii) gives the desired result.\\

{\em Step 2}. Now we turn to the embedding $\id : \MA (\Omega) \hookrightarrow L_{\infty}(\Omega)$ and make use of the above embeddings to the $B$-scale. By \eqref{elem-tau} and \eqref{fte} with $0\leq \tau=\frac1p-\frac1u$, we have
\[
B^{s, \tau}_{p, \min\{p,q\}} (\Omega) \hookrightarrow \MF(\Omega)= \ft(\Omega) \hookrightarrow B^{s, \tau}_{p, \max\{p,q\}}(\Omega)\hookrightarrow 
B^0_{\infty, 1}(\Omega) \hookrightarrow L_{\infty}(\Omega) \hookrightarrow B^0_{\infty, \infty}(\Omega).
\]
According to Theorem~\ref{ek-tau}(ii), we then get
\[
e_k(\id : \MF (\Omega) \hookrightarrow L_{\infty} (\Omega)) \ \lesssim \ e_k(\id_\tau :B^{s, \tau}_{p, \max\{p,q\}}(\Omega) \hookrightarrow  B^{0}_{\infty, 1}(\Omega)) \ \lesssim \ k^{-\frac{s}{d}}, \quad k \in \nn,
\]
and 
\[
e_k(\id : \MF (\Omega) \hookrightarrow L_{\infty} (\Omega)) \ \gtrsim \ e_k(\id_\tau :B^{s, \tau}_{p, \min\{p,q\}}(\Omega) \hookrightarrow  B^{0}_{\infty, \infty}(\Omega)) \ \gtrsim \ k^{-\frac{s}{d}}, \quad k \in \nn,
\]
whenever $s> \frac{d}{u}$, which completes the proof for $\mathcal{A}=\mathcal{E}$.

Similarly we obtain the result for $\mathcal{A}=\mathcal{N}$. Namely, we have
\[
e_k \left(\id _{\mathcal{N}}:\MB (\Omega) \hookrightarrow \mathcal{N}^{0}_{\infty, \infty, \infty}(\Omega) \right)\ \lesssim \ e_k( \id)\ \lesssim \ e_k \left(\id _{\mathcal{N}}:\mathcal{N}^{s}_{u,p, \infty}(\Omega) \hookrightarrow \mathcal{N}^{0}_{\infty, \infty, 1}(\Omega) \right),
\]
which in view of Theorem~\ref{ek-NE}(ii) gives the result. 
}

\end{proof}


\begin{remark}
Corollary~\ref{ek-Morrey-Linfty} extends some result on entropy numbers for the target space $L_\infty(\Omega)$ obtained in \cite{HaSk-krakow}. There we could only cover the case $s>\frac{d}{p}$. 
\end{remark}

In a parallel way we can further characterise the compactness of the embeddings described by Corollary~\ref{atauinLinf}.

\begin{corollary}\label{ek-atau-Linf}
 Let $s\in\real$, $\tau\geq 0$, $0< p,q\leq\infty$ (with $p<\infty$ if $A=F$), and assume $s> d(\frac1p-\tau)$. Then the entropy numbers of the compact embeddings $\id_\tau: \at(\Omega) \hookrightarrow L_{\infty}(\Omega)$ and $\id_\tau: \at(\Omega )\hookrightarrow \bmo(\Omega)$ behave like
\[
e_k\left(\id:\at(\Omega) \hookrightarrow L_{\infty}(\Omega)\right)\ \sim \ e_k\left(\id: \at(\Omega )\hookrightarrow \bmo(\Omega)\right) \ \sim \ k^{-\frac{s}{d}-(\tau-\frac1p)_+},\quad k\in\nn.
\]
\end{corollary}

\begin{proof}
We proceed as above and observe that Theorem~\ref{ek-tau}(i) is again excluded for the same reasons.
\end{proof}

Finally we deal with the target space $L_r(\Omega)$, $1\leq r<\infty$. First we recall our result obtained in \cite[Corollary~4.8]{HaSk-morrey-comp}.

\begin{corollary}\label{C-ekintoLr}
  Let $1\leq r<\infty$, $0<p\leq u<\infty$, or $p=u=\infty$, and $s>d \frac{p}{u}\left(\frac1p-\frac1r\right)_+$.
\bli
\item[{\upshape\bfseries (i)}] If $ p\geq r$ and $s>0$, or $p<r$ and $s>d(\frac1p-\frac1r)$, then
\begin{equation}
    e_k \big(\id: \MA(\Omega)\hookrightarrow L_r(\Omega)\big) \ \sim \  k^{-\frac{s}{d}} .
\end{equation}
\item[{\upshape\bfseries (ii)}]
  If $p<r$ and $\ d \frac{p}{u} \Big(\frac{1}{p}- \frac{1}{r}\Big)< s \le  d\Big(\frac{1}{p} -\frac{1 }{r}\Big)$, then there exists some $c>0$ and for any $\varepsilon>0$ some $c_\varepsilon>0$ such that for all $k\in\nn$,
\begin{equation}\label{new_morrey_Lr}
  c \ k^{- \frac{u}{u-p}(\frac{s}{d}-\frac{p}{u}(\frac{1}{p}- \frac{1}{r}) )}\ \le \ e_k \big(\id: \MA(\Omega)\hookrightarrow L_r(\Omega)\big)\ \le \ c_\varepsilon \ k^{- \frac{u}{u-p}(\frac{s}{d}-\frac{p}{u}(\frac{1}{p}- \frac{1}{r}) )+\varepsilon}.
\end{equation}
\eli
\end{corollary}

Now we can strengthen our above compactness result Corollary~\ref{atauinLr} as follows.

\begin{corollary}\label{ek-atau-Lr}
 Let $s\in\real$, $0< p,q\leq\infty$ (with $p<\infty$ if $A=F$), $1\leq r<\infty$, $\tau> 0$. Assume that \eqref{lim-12} is satisfied.
\bli
\item[{\upshape\bfseries (i)}]
If
\[\tau<\frac1p,\quad  p<r, \quad \text{and}\quad  (1-p\tau)\left(\frac1p-\frac1r\right)<\frac{s}{d}<\frac1p-\frac1r,
\]
 then there exists some $c>0$ and for any $\varepsilon>0$ some $c_\varepsilon>0$ such that for all $k\in\nn$,
\[
  c \ k^{- \frac{1}{p \tau}(\frac{s}{d}-(1-p\tau)(\frac{1}{p}- \frac{1}{r}) )} \ \le \ e_k \left(\id :\at(\Omega)\hookrightarrow L_r(\Omega)\right)\ \le \ c_\varepsilon \ k^{- \frac{1}{p\tau}(\frac{s}{d}-(1-p\tau)(\frac{1}{p}- \frac{1}{r}) )+\varepsilon}.
\]
\item[{\hfill\bfseries\upshape (ii)\hfill}]
In all other cases admitted by \eqref{lim-12}, it holds
\[
e_k\left(\id :\at(\Omega)\hookrightarrow L_r(\Omega)\right) \ \sim \ k^{-\frac{s}{d} - (\tau-\frac{1}{p})_+ },\quad k\in\nn.
\]
\eli
\end{corollary}

\begin{proof}
We apply Theorem~\ref{ek-tau} and follow otherwise exactly the same line of arguments, in particular, with $B^0_{r,1}(\Omega)\hookrightarrow L_r(\Omega)\hookrightarrow B^0_{r,\infty}(\Omega)$.
\end{proof}

\begin{remark}
Note that we (only) have an influence of $r$ when we are in the proper Morrey case ($0<\tau<\frac1p$) and $s$ is small enough.
\end{remark}

\begin{remark} Due to their similarities, we do not present the special cases when the source or the target space matches the classical spaces $\B$ and $\F$, that is when $\tau_1=0$ or $\tau_2=0$. However, we would like to remark that the result is not symmetric in the sense, that we have different results for both cases. Namely, when $\tau_1=0$, part (i) of Theorem \ref{ek-tau} is excluded, while for the case where $\tau_2=0$ 
the both parts of the theorem are relevant, naturally with the proper adaptations for this particular case.
\end{remark}

\section{Approximation numbers}
Finally we briefly collect some partial results about approximation numbers of the embedding $\id_\tau$, recall their definition \eqref{approx-n}. In \cite{YHMSY} we obtained some first result for approximation numbers $a_k(\id_\tau)$ when the target space was $L_\infty(\Omega)$:
Let $p\in[2,\infty]$ $($with $p<\infty$ in the $F$-case$)$, $q\in(0,\infty]$, $0\leq \tau<\frac1p$ and
$d(\frac1p-\tau)< s < d(\frac1p-\tau)+1$. Then
\begin{equation}
a_k\left(\id : \at(\Omega) \to L_\infty(\Omega)\right) \sim k^{-\frac{s}{d} -\tau + \frac1p}, \quad k\in\nn.
\end{equation}

In \cite{HaSk-krakow} we studied the situation for the embedding $\id_{\mathcal{A}}$ with  the following result.

\begin{proposition}\label{ak_id_A}
  Let  $s_i\in \real$, $0<q_i\leq\infty$, $0<p_i\leq u_i<\infty$, or $p_i=u_i=\infty$, $i=1,2$, with
\begin{equation}\label{ak-3}
s_1>s_2,\quad\text{and}\quad p_1\geq u_2.
\end{equation}
Then
\begin{equation}
  a_k\left( \id_{\mathcal{A}}: \MAe(\Omega )\hookrightarrow \MAz(\Omega )\right)\sim k^{-\frac{s_1-s_2}{d}},\quad k\in\nn.
\end{equation}
\end{proposition}

The above proposition coincides with \cite[Corollary~3.4(i)]{HaSk-krakow} apart from the case when $p_i=u_i=\infty$ for $i=1$ or $i=2$. But this extension can easily be verified following the short proof in \cite{HaSk-krakow}. We also refer to \cite[Section~6]{Bai-Si} where also the periodic case and more general Morrey type spaces were studied. \\

Now we give some partial counterpart of Theorem~\ref{ek-tau} in terms of approximation numbers.

\begin{corollary}  \label{ak-tau}
  Let  $s_i\in \real$, $0<q_i\leq\infty$, $0<p_i\leq \infty$ (with $p_i<\infty$ in case of $A=F$), $\tau_i\geq 0$, $i=1,2$.
  Assume that $\ s_1-s_2 > d \ \critical$ and, in addition,
\bli
\item[{\hfill\bfseries\upshape (i)\hfill}]
either $\ \tau_1\geq \frac{1}{p_1}$,
\item[{\hfill\bfseries\upshape (ii)\hfill}]
or $\ \tau_i<\frac{1}{p_i}$, $i=1,2$, with $s_1>s_2$ and $\tau_2\leq \frac{1}{p_2}-\frac{1}{p_1}$.
  \eli
  Then we obtain for the approximation numbers of the compact embedding
\begin{equation*}
	 \id_{\tau}: \ate(\Omega )\hookrightarrow \atz(\Omega )
\end{equation*}
that
\begin{equation}\label{ak-2}
a_k(\id_\tau) \ \sim \ k^{-\frac{s_1-s_2}{d} - (\tau_1-\frac{1}{p_1})_+ + (\tau_2-\frac{1}{p_2})_+},\quad k\in\nn.
\end{equation}
\end{corollary}

\begin{proof}
  {\em Step 1}.~ We begin with case (i) and assume that $\tau_1\geq \frac{1}{p_1}$, with $q_1=\infty$ when $\tau_1=\frac{1}{p_1}$. In case of $\tau_2\geq \frac{1}{p_2}$ with $q_2=\infty$ when $\tau_2=\frac{1}{p_2}$, then we use classical Besov space results,
  \begin{align*}
    a_k\left(\id_\tau:\ate(\Omega)\hookrightarrow \atz(\Omega)\right) \sim & \ a_k\left(\id_B: B^{s_1+d(\tau_1-\frac{1}{p_1})}_{\infty,\infty}(\Omega) \hookrightarrow B^{s_2+d(\tau_2-\frac{1}{p_2})}_{\infty,\infty}(\Omega)\right) \\
    \sim &\ k^{-\frac{s_1-s_2}{d} - (\tau_1-\frac{1}{p_1}) + (\tau_2-\frac{1}{p_2})}, \quad k \in \nn,
  \end{align*}
in view of Proposition~\ref{yy02} and the corresponding approximation number result  \eqref{ak-id_A}. This proves \eqref{ak-2} in this case.

If $0\leq \tau_2<\frac{1}{p_2}$, then we proceed similar to Substep~ 2.2 of the proof of Theorem~\ref{ek-tau} and make use of the chains of embeddings \eqref{tau2small-3} and \eqref{tau2small-4}. Therefore, we get
\[
a_k\left(\id_{\mathcal{N}}: \mathcal{N}^{s_1+d(\tau_1-\frac{1}{p_1})}_{\infty,\infty,\infty}(\Omega)\hookrightarrow \mathcal{N}^{s_2}_{u_2, p_2, \infty}(\Omega)\right) \ \lesssim \ a_k(\id_\tau) \ \lesssim \ a_k\left(\id: \mathcal{N}^{s_1+d(\tau_1-\frac{1}{p_1})}_{\infty,\infty,\infty}(\Omega)\hookrightarrow \mathcal{A}^{s_2}_{u_2, p_2, q_2}(\Omega)\right), 
\]
which, in view of Proposition~\ref{ak_id_A}, leads to 
\[
a_k(\id_\tau) \ \sim \ k^{-\frac{s_1-s_2}{d}-(\tau_1-\frac{1}{p_1})}, \quad k\in\nn. 
\]

As for the case $\tau_2=\frac{1}{p_2}$, $q_2<\infty$, we can follow Substep~3.4 of the proof of Theorem~\ref{ek-tau} and apply Proposition~\ref{ak_id_A} instead of Theorem~\ref{ek-NE}.\\

{\em Step 2}.\quad Assume that $\tau_1=\frac{1}{p_1}$ with $q_1<\infty$. Again we check our above proof of Theorem~\ref{ek-tau} and find that Substep~3.1 with Proposition~\ref{ak_id_A} cover the case $0<\tau_2<\frac{1}{p_2}$, while Substep~3.2 of that proof is related to the case $\tau_2\geq\frac{1}{p_2}$ with $q_2=\infty$ if $\tau_2=\frac{1}{p_2}$, and Substep~3.5 concerns the situation when   $\tau_2=\frac{1}{p_2}$ with $q_2<\infty$. In the latter two cases we benefit from  \eqref{ak-id_A}.\\

{\em Step 3}.~ We deal with (ii). In this case, we can apply \eqref{fte} and Proposition~\ref{ak_id_A} to obtain the result for $\at=\ft$. Otherwise we argue similarly as in Step~1 and, using \eqref{tau2small-1} and \eqref{tau2small-2} this time, we get
\[
a_k\left(\id_{\mathcal{N}}: \mathcal{N}^{s_1}_{u_1,p_1,q_1}(\Omega)\hookrightarrow \mathcal{N}^{s_2}_{u_2,p_2,\infty}(\Omega)\right)\ \lesssim \ a_k(\id_\tau)\ \lesssim \ a_k\left(\id_{\mathcal{N}}: \mathcal{N}^{s_1}_{u_1,p_1,\infty}(\Omega) \hookrightarrow \mathcal{N}^{s_2}_{u_2,p_2,q_2}(\Omega)\right),
\]
where $\frac{1}{u_i}=\frac{1}{p_i}-\tau_i$, $i=1,2$, as usual. Now we apply Proposition~\ref{ak_id_A} (with $\mathcal{A}=\mathcal{N}$), under the additional assumptions made in (ii), and benefit again from the independence of $q_i$.
\end{proof}

\begin{remark}
It is obvious from the above proof, in particular in Steps~1 and 2, that we could follow all the arguments of the proof of Theorem~\ref{ek-tau} to deal with the remaining cases. However, in view of the additional restriction \eqref{ak-3} in Proposition~\ref{ak_id_A} and the observation of $p$-dependence in \eqref{ak-id_A} this leads to partial one-sided results only.
\end{remark}

\begin{remark}
  Note that one could reformulate Corollary~\ref{ak-tau} in terms of the hybrid spaces $L^r\A(\Omega)$  in the spirit of Remark~\ref{ek-tau}. We leave it to the reader.
  \end{remark}

Finally we conclude a few special cases from Corollary~\ref{ak-tau} and start with the case $\tau_1=\tau_2=\tau$.

\begin{corollary}  \label{ak-tau-tau}
  Let  $s_i\in \real$, $0<q_i\leq\infty$, $0<p_i\leq \infty$ (with $p_i<\infty$ in case of $A=F$),  $i=1,2$, and $\tau\geq 0$. Assume that \eqref{same-tau} is satisfied and, in addition,
\bli
\item[{\hfill\bfseries\upshape (i)\hfill}]
either $\ \tau\geq \frac{1}{p_1}$,
\item[{\hfill\bfseries\upshape (ii)\hfill}]
or $\ \tau<\min\left\{\frac{1}{p_1},\frac{1}{p_2}\right\}$, with $\tau\leq \frac{1}{p_2}-\frac{1}{p_1}$.
  \eli
  Then we obtain for the approximation numbers of the compact embedding
\begin{equation*}
	 \id_{\tau}: A^{s_1,\tau}_{p_1,q_1}(\Omega )\hookrightarrow A^{s_2,\tau}_{p_2,q_2}(\Omega )
\end{equation*}
that
\begin{equation}\label{ak-4}
  a_k(\id_\tau) \ \sim \ \begin{cases}
  k^{-\frac{s_1-s_2}{d} +\frac{1}{p_1}-\frac{1}{p_2}}, & \text{if}\quad \tau\geq \max\left\{\frac{1}{p_1},\frac{1}{p_2}\right\}, \\[1ex]
    k^{-\frac{s_1-s_2}{d} - \tau +\frac{1}{p_1}},& \text{if}\quad \frac{1}{p_2}>\tau\geq \frac{1}{p_1},\\[1ex]
    k^{-\frac{s_1-s_2}{d}}, & \text{if}\quad \tau<  \min\left\{\frac{1}{p_1},\frac{1}{p_2}\right\}\quad\text{and}\quad \tau\leq \frac{1}{p_2}-\frac{1}{p_1}.
\end{cases}
  \end{equation}
\end{corollary}

Next we consider the special target spaces $L_r(\Omega)$, $1\leq r\leq\infty$, and $\bmo(\Omega)$.

\begin{corollary}\label{ak-atau-Linf}
  Let $s\in\real$, $0< p,q\leq\infty$ (with $p<\infty$ if $A=F$), and assume that $\tau\geq \frac1p$ and
  $s> d(\frac1p-\tau)$. Then the approximation numbers of the compact embeddings $\id: \at(\Omega) \hookrightarrow L_{\infty}(\Omega)$ and $\id: \at(\Omega )\hookrightarrow \bmo(\Omega)$ behave like
\[
a_k\left(\id:\at(\Omega) \hookrightarrow L_{\infty}(\Omega)\right)\sim a_k\left(\id: \at(\Omega )\hookrightarrow \bmo(\Omega)\right) \sim k^{-\frac{s}{d}-\tau+\frac1p},\quad k\in\nn.
\]
\end{corollary}

\begin{remark}
In our paper \cite{YHMSY} we obtained already that
\[
a_k\left(\id:\at(\Omega) \hookrightarrow L_{\infty}(\Omega)\right)\sim  k^{-\frac{s}{d}-\tau+\frac1p},\quad k\in\nn,
\]
if $0<p,q\leq\infty$ (with $p<\infty$ if $A=F$), $(\frac1p-\frac12)_+\leq \tau<\frac1p$, and  $d(\frac1p-\tau)<s<d(\frac1p-\tau)+1$ (please note the misprints in \cite[Corollary~5.10]{YHMSY}). So the above result can be seen as some partial extension to the case when $\tau\geq \frac1p$.
\end{remark}

The partial counterpart of Corollary~\ref{ek-atau-Lr} reads as follows.

\begin{corollary}\label{ak-atau-Lr}
 Let $s\in\real$, $0< p,q\leq\infty$ (with $p<\infty$ if $A=F$), $1\leq r<\infty$, $\tau\geq 0$. Assume that
 \eqref{lim-12} is satisfied, and, in addition, that
\bli
\item[{\hfill\bfseries\upshape (i)\hfill}]
either $\ \tau\geq \frac{1}{p}$,
\item[{\hfill\bfseries\upshape (ii)\hfill}]
or $\ \tau<\frac1p\leq \frac1r$.
  \eli
  Then we obtain that
\[
a_k(\id:\at(\Omega)\hookrightarrow L_r(\Omega)) \ \sim \ k^{-\frac{s}{d} - (\tau-\frac{1}{p})_+ },\quad k\in\nn.
\]
\end{corollary}

\begin{remark}
  Note that Corollaries~\ref{ak-atau-Linf} and \ref{ak-atau-Lr} deal with situations where entropy and approximation numbers show the same asymptotic behaviour, cf. Corollaries~\ref{ek-atau-Linf} and \ref{ek-atau-Lr}, which is in general not the case.
\end {remark}


\vspace{1cm}
\noindent
Helena F.~Gon\c{c}alves\\
Institute of Mathematics, Friedrich-Schiller-University Jena, 07737 Jena, Germany\\ E-mail: helena.goncalves@uni-jena.de \\\\
\noindent
Dorothee D.~Haroske \\
Institute of Mathematics, Friedrich-Schiller-University Jena, 07737 Jena, Germany\\
E-mail: dorothee.haroske@uni-jena.de \\\\
\noindent
Leszek Skrzypczak \\
Faculty of Mathematics and Computer Science, Adam Mickiewicz University,\\ ul. Uniwersytetu Pozna\'nskiego 4, 61-614 Poznan, Poland \\
E-mail: lskrzyp@amu.edu.pl

\end{document}